\documentclass[11pt,reqno]{amsart} 
\usepackage[applemac]{inputenc}
\usepackage{amsmath}
\usepackage{amssymb}
\usepackage{euscript}
\usepackage{mathrsfs}
\usepackage{wasysym}
\usepackage[all]{xy}
\usepackage{graphicx}
\usepackage{color}
\usepackage{multirow}
\usepackage{pifont}
\usepackage[table]{xcolor}
\usepackage[colorlinks=true,linktocpage=true,pagebackref=false, urlcolor=black, citecolor=black,linkcolor=black]{hyperref}

\newtheorem{thm}{Theorem}[section]

\newtheorem{lem}[thm]{Lemma}
\newtheorem{prop}[thm]{Proposition}

\newtheorem{cor}[thm]{Corollary}

\theoremstyle{definition}

\theoremstyle{remark}
\newtheorem{remark}[thm]{Remark}

\newtheorem{example}[thm]{Example}
\newtheorem{examples}[thm]{Examples}

\numberwithin{equation}{section}

\newcommand{\K}{{\mathbb K}} 
 \newcommand{\R}{{\mathbb R}}
 \newcommand{\C}{{\mathbb C}}

\newcommand{\sph}{{\mathbb S}}

\newcommand{\an}{{\mathcal O}} 
 
 \newcommand{\J}{{\mathcal J}}
 
\newcommand{\mer}{{\mathcal M}} 
\newcommand{\ideal}{{\mathcal I}}

\newcommand{\gtp}{{\mathfrak p}} 
\newcommand{\gtm}{{\mathfrak m}} \newcommand{\gtn}{{\mathfrak n}}
\newcommand{\gta}{{\mathfrak a}}

\newcommand{\gtg}{{\mathfrak g}}

\newcommand{\Fhaz}{{\EuScript F}}
\newcommand{\Jhaz}{{\EuScript I}}

\newcommand{\Ss}{{\EuScript S}}
\newcommand{\Tt}{{\EuScript T}}

\newcommand{\Reg}{\operatorname{Reg}}
\newcommand{\Sing}{\operatorname{Sing}}

\newcommand{\Int}{\operatorname{Int}}
\newcommand{\cl}{\operatorname{Cl}}

\newcommand{\id}{\operatorname{id}}

\newcommand{\zar}{\operatorname{zar}}

\newcommand{\x}{{\tt x}} \newcommand{\y}{{\tt y}} 
\newcommand{\z}{{\tt z}} 
 \renewcommand{\u}{{\tt u}}

\newcommand{\veps}{\varepsilon}

\newcommand{\ol}{\overline}

\numberwithin{equation}{section}

\begin{document}
\title[On globally defined semianalytic sets]{On globally defined semianalytic sets}

\author{Francesca Acquistapace}
\author{Fabrizio Broglia}
\address{Dipartimento di Matematica, Universit\`a degli Studi di Pisa, Largo Bruno Pontecorvo, 5, 56127 
PISA (ITALY)}
\email{acquistf@dm.unipi.it, broglia@dm.unipi.it}

\author{Jos\'e F. Fernando}
\address{Departamento de \'Algebra, Facultad de Ciencias Matem\'aticas, Universidad Complutense de Madrid, 28040 MADRID (SPAIN)}
\email{josefer@mat.ucm.es}

\date{28/02/2015}

\thanks{Authors supported by Spanish GAAR MTM2011-22435, Spanish MTM2014-55565 and the ``National Group for Algebraic and Geometric Structures, and their Applications'' (GNSAGA - INdAM). First and second authors are also supported by Italian MIUR. Third author is also supported by Grupos UCM 910444. This article is the fruit of the close collaboration of the authors in the last ten years and has been mainly written during a one-year research stay of third author in the Dipartimento di Matematica of the Universit\`a di Pisa. Third author would like to thank the department for the invitation and the very pleasant working conditions. The one-year research stay of the third author is partially supported by MECD grant PRX14/00016.}

\begin{abstract}--- \ 
In this work we present the concept of \em $C$-semianalytic subset \em of a real analytic manifold and more generally of a real analytic space. $C$-semianalytic sets can be understood as the natural generalization to the semianalytic setting of global analytic sets introduced by Cartan ($C$-analytic sets for short). More precisely $S$ is a $C$-semianalytic subset of a real analytic space $(X,\an_X)$ if each point of $X$ has a neighborhood $U$ such that $S\cap U$ is a finite boolean combinations of global analytic equalities and strict inequalities on $X$. By means of paracompactness $C$-semianalytic sets are the locally finite unions of finite boolean combinations of global analytic equalities and strict inequalities on $X$.

The family of $C$-semianalytic sets is closed under the same operations as the family of semianalytic sets: locally finite unions and intersections, complement, closure, interior, connected components, inverse images under analytic maps, sets of points of dimension $k$, etc. although they are defined involving only global analytic functions. In addition, we characterize subanalytic sets as the images under proper analytic maps of $C$-semianalytic sets.

We prove also that \em the image of a $C$-semianalytic set $S$ under a proper holomorphic map between Stein spaces is again a $C$-semianalytic set\em. The previous result allows us to understand better the structure of the set $N(X)$ of points of non-coherence of a $C$-analytic subset $X$ of a real analytic manifold $M$. We provide a global geometric-topological description of $N(X)$ inspired by the corresponding local one for analytic sets due to Tancredi-Tognoli (1980), which requires complex analytic normalization. As a consequence it holds that \em $N(X)$ is a $C$-semianalytic set of dimension $\leq\dim(X)-2$\em. 
\end{abstract}
\maketitle

\section{Introduction}\label{s1}

Let $M$ be a real analytic manifold. A subset $X$ of $M$ is \em (real) analytic \em if for each point $x\in M$, there exists an open neighborhood $U^x$ such that $X\cap U^x=\{f_1=0,\ldots,f_r=0\}\subset U^x$ for some $f_1,\ldots,f_r\in\an(U^x)$. Already in the 1950s Cartan, Whitney and Bruhat noticed that this class of sets, whose definition is a straight conversion to the real case of the concept of complex analytic subset of a complex analytic manifold, does not enjoy all the good properties of complex analytic sets (for instance: coherence, irreducible components, etc.). The global behavior of these sets could be wild as it is shown in the exotic examples presented in \cite{bc,c,wb}. Concerning this fact Cartan wrote in \cite[pag. 49]{c2} the following:

{\em ``\ldots la seule notion de sous-ensemble analytique r\'eel (d'une vari\'et\'e analytique-r\'eelle $V$) qui ne conduise pas \`a des propri\'et\'es pathologiques doit se r\'ef\'erer \`a l'espace complexe ambiant: il faut consid\'erer les sous-ensembles ferm\'es $E$ de $V$ tels qu'il existe une complexification $X$ de $V$ et un sous-ensemble analytique-complexe $E'$ de $W$, de mani\`ere que $E = W \cap E'$. On d\'emontre que ce sont aussi les sous-ensembles de $V$ qui peuvent \^etre d\'efinis globalement par un nombre fini d'\'equations analytiques. La notion de sous-ensemble analytique-r\'eel a ainsi un caract\`ere essentiellement global, contrairement \`a ce qui avait lieu pour les sous-ensembles analytiques-complexes.'' }

The special class of real analytic subsets of a real analytic manifold $M$ introduced by Cartan \cite{c} is the family of {\em analytic subsets of $M$ that can be described by finitely many global analytic equations}. They are commonly known as \em $C$-analytic sets\em. Of course, such class contains classical coherent analytic sets but also more general ones as Whitney's umbrella. Naturally, for each point $x$ of a real analytic set $S$ there exists a small open neighborhood $U\subset M$ of $x$ such that the intersection $S\cap U$ is a $C$-analytic subset of $U$.

In Real Geometry also appear naturally sets described by inequalities. A subset $S$ of a real analytic manifold $M$ is a \em semianalytic set \em if for each point $x\in M$ there exists an open neighborhood $U^x$ such that $S\cap U^x$ is a finite union of sets of the type $\{f=0,g_1>0,\ldots,g_r>0\}\subset U^x$ where $f,g_i\in\an(U^x)$ are analytic functions on $U^x$. 

Semianalytic sets (and more generally subanalytic sets) were introduced by \L ojasiewicz in \cite{l,l1} and were developed later by many authors: Bierstone-Milman \cite{bm,bm2}, Hironaka \cite{hi1,hi2,hi3,hi4}, Gabrielov \cite{ga}, Hardt \cite{h1,h2}, Galbiati \cite{gal}, Paw\l ucki \cite{pa}, Denkowska \cite{de}, Stasica \cite{s}, Kurdyka \cite{k}, Parusi\'nski \cite{p}, Shiota \cite{sh} between others. These sets have many and wide applications in complex and real analytic geometry.

While the family of complex analytic sets is stable under proper holomorphic maps between complex analytic spaces (Remmert's Theorem \cite[VII.\S2.Thm.2]{n}), an analogous property does not hold in the real analytic setting. The image of a real analytic set under a proper real analytic map is not even in general a semianalytic set. Indeed, this fact promoted the introduction of subanalytic sets by \L ojasiewicz \cite{l} in the 1960s.

In \cite{gal2} Galbiati proved that if $f:X\to Y$ is a proper analytic map between real analytic spaces that admits a proper complexification $\widetilde{f}:\widetilde{X}\to\widetilde{Y}$ and $Z$ is a $C$-analytic subset of $X$, then $f(X\setminus Z)$ is a semianalytic set. In \cite{hi4} Hironaka quoted this result and remarked that $f(X\setminus Z)$ is \em `globally semianalytic in $Y$ with respect to the given complexification $\widetilde{Y}$ of $Y$'\em. In \cite[p.404]{tt} it is used without proof that the image of a semianalytic set under a proper invariant holomorphic map between Stein spaces is semianalytic. Although it seemed to be an assumed result in the 1970s, we have found no precise reference with a proof to this fact.

At this point, it is natural to wonder whether it is possible to find, amalgamating the concepts of $C$-analytic sets and semianalytic sets, a family of semianalytic sets `globally defined' in the sense of Cartan and Hironaka that enjoy a good behavior with respect to basic boolean, topological and algebraic operations and admit a kind of direct image theorem for good enough proper analytic maps. In addition, assume that the set of points of a $C$-analytic set satisfying a property $P$ is semianalytic. We would like to understand when $P$ provides a globally defined semianalytic set.

A global approach to semianalytic sets was explored first by Andradas-Br\"ocker-Ruiz \cite{abr2,rz1,rz2,rz3} under compactness assumptions and by Andradas-Castilla \cite{ac} for low dimension. They defined a {\em global semianalytic subset $S$ of a real analytic manifold $M$} as a finite union of global basic semianalytic sets, that is, a finite union of sets of the type $\{f=0,g_1>0,\ldots,g_r>0\}$ where $f,g_j\in\an(M)$. As far as we know, it is not known whether the family of global semianalytic sets is closed under taking closure, interior or connected components except for $\dim(M)\leq 2$. There exist further information concerning closure (and interior) of a global semianalytic set if $\dim(M)=3$ but nothing conclusive for higher dimension if the involved global semianalytic set has non-compact boundary. 

In this work we refine the previous approach and consider the family of {\it $C$-semianalytic sets}, which is constituted by those semianalytic sets that are locally finite unions of global basic semianalytic sets (see Section \ref{s3}). As $M$ is paracompact, $C$-semianalytic sets are those semianalytic sets of $M$ that admit in a small neighborhood of each point of $M$ a local description involving only global analytic functions on $M$ (see Lemma \ref{cover}). Roughly speaking, we are considering the class of those semianalytic sets that can be described using only elements of $\an(M)$. 

We see in this article that the family of $C$-semianalytic sets enjoys the desired properties. However, it is still too large to have a consistent concept of irreducibility and a reasonable theory of irreducible components and requires a refinement to approach this problem. To that end in \cite{fe1} it is introduced the subfamily of `amenable' $C$-semianalytic sets and it is developed a satisfactory theory of irreducible components.

\subsection{Main results}

We present next the main results of this work.

\subsubsection{Main properties of $C$-semianalytic sets}

The family of $C$-semianalytic sets is closed under the same operations as semianalytic sets: locally finite unions and intersections, complement, closure, interior, connected components, sets of points of dimension $k$ and inverse images of analytic maps. Analogously to the semianalytic and semialgebraic cases: \em open $C$-semianalytic sets and closed $C$-semianalytic sets admits local homogeneous descriptions \em (see Lemma \ref{hop}). There exist many semianalytic sets that are not $C$-semianalytic sets, we refer the reader to Examples \ref{counterexamples}. In addition $C$-semianalytic sets satisfy the following type of complex-proper image theorem that we prove in Section \ref{s4}.

\subsubsection{Images of $C$-semianalytic sets under proper holomorphic maps}
Let $(X,\an_X)$ and $(Y,\an_Y)$ be reduced Stein spaces. Let $\sigma:X\to X$ and $\tau:Y\to Y$ be anti-involutions. Assume 
$$
X^\sigma:=\{x\in X:\ x=\sigma(x)\}\quad\text{and}\quad Y^\tau:=\{y\in Y:\ y=\tau(y)\}
$$ 
are non-empty sets. It holds that $(X^\sigma,\an_{X^\sigma})$ and $(Y^\tau,\an_{Y^\tau})$ are real analytic spaces. Observe that $(X,\an_X)$ and $(Y,\an_Y)$ are complexifications of $(X^\sigma,\an_{X^\sigma})$ and $(Y^\tau,\an_{Y^\tau})$. 

We will say that a $C$-semianalytic set $S\subset X^\sigma$ is \em ${\mathcal A}(X^\sigma)$-definable \em if for each $x\in X^\sigma$ there exists an open neighborhood $U^x$ such that $S\cap U^x$ is a finite union of sets of the type $\{F|_{X^\sigma}=0,G_1|_{X^\sigma}>0,\ldots,G_r|_{X^\sigma}>0\}$ where $F,G_i\in\an(X)$ are invariant holomorphic sections. We denote the set of $\sigma$-invariant holomorphic functions of $X$ restricted to $X^\sigma$ with ${\mathcal A}(X^\sigma)$.

\begin{thm}[Direct image theorem]\label{properint}
Let $F:(X,\an_X)\to(Y,\an_Y)$ be an invariant proper holomorphic map, that is, $\tau\circ F=F\circ\sigma$. Let $S\subset X^\sigma$ be an ${\mathcal A}(X^\sigma)$-definable $C$-semianalytic set. We have 
\begin{itemize}
\item[(i)] $F(S)$ is a $C$-semianalytic subset of $Y^\tau$ of the same dimension as $S$. 
\item[(ii)] Let $E:=\cl(F^{-1}(Y^\tau)\setminus X^\sigma)$. Then $F(E\cap S)$ is a $C$-semianalytic subset of $Y^\tau$.
\item[(iii)] If $S$ is a $C$-analytic set and $F^{-1}(Y^\tau)=X^\sigma$, then $F(S)$ is also a $C$-analytic set.
\end{itemize}
\end{thm}
\begin{remark}
Let $S$ be a semianalytic subset of $X^\sigma$. For each $x\in X^\sigma$ there exists an open neighborhood $U\subset X^\sigma$ of $x$ such that $S\cap U$ is a $C$-semianalytic subset of $U$. Using this fact and Theorem \ref{properint} one shows straightforwardly that the image of a semianalytic set under an invariant proper holomorphic map between reduced Stein spaces is a semianalytic set.
\end{remark}

The keys to prove Theorem \ref{properint} are \cite[Thm.2.2]{bm} (whose statement is recalled in \ref{bmp}) and the following result that analyzes the local structure of proper surjective holomorphic morphisms between Stein spaces. For each $x\in X$ we denote the maximal ideal of $\an(X)$ associated to $x$ with $\gtm_x$ and for each $y\in Y$ we denote the maximal ideal of $\an(Y)$ associated to $y$ with $\gtn_y$. Recall that compact analytic subsets of a Stein space are finite sets, so the fibers of a proper holomorphic map between Stein spaces are finite sets. Write $F^*(\an(Y)):=\{G\circ F:\ G\in\an(Y)\}\subset\an(X)$ and 
$$
F^*(\an(Y)_{\gtn_y})=\Big\{\frac{G\circ F}{H\circ F}:\ G,H\in\an(Y)\ \text{and}\ H\notin\gtn_y\Big\}.
$$

\begin{thm}[Local structure of finite holomorphic morphisms]\label{finite76}
Let $F:(X,\an_X)\to(Y,\an_Y)$ be a surjective proper holomorphic map between reduced Stein spaces and let $y_0\in Y$. Denote the multiplicative subset of $\an(X)$ constituted by all the holomorphic functions on $X$ that do not vanish at the set $F^{-1}(y_0):=\{x_1,\ldots,x_\ell\}$ with $\Ss:=\an(X)\setminus(\gtm_{x_1}\cup\cdots\cup\gtm_{x_\ell})$. Then $\Ss^{-1}(\an(X))$ is a finitely generated $\an(Y)_{\gtn_{y_0}}$-module and there exist invariant holomorphic functions $H_1,\ldots,H_m\in\an(X)$ such that 
$$
\Ss^{-1}(\an(X))=F^*(\an(Y)_{\gtn_{y_0}})[H_1,\ldots,H_m].
$$
\end{thm}

\subsubsection{$C$-properties} 
Let $P$ be a property concerning either $C$-semianalytic or $C$-analytic sets. We say that $P$ is a {\em $C$-property} if the set of points of an either $C$-semianalytic or $C$-analytic set $S$ satisfying $P$ is a $C$-semianalytic set. For example, as we have already commented, the set of points where the dimension of the $C$-semianalytic set $S$ is $k$ is again a $C$-semianalytic set. As a consequence of Theorem \ref{properint} we prove in Section \ref{s5} that the set of points of non-coherence of a $C$-analytic set is $C$-semianalytic, that is, `to be a point of non-coherence' (or `to be a point of coherence') are $C$-properties.

\subsubsection{Set of points of non-coherence}
The set of points $N(X)$ where an analytic set $X\subset M$ is not coherent was studied first by Fensch in \cite[I.\S2]{fe} where he proved that \em it is contained in a semianalytic set of dimension $\leq\dim(X)-2$\em. This result was revisited by Galbiati in \cite{gal} and she proved that \em it is in fact a semianalytic set\em. Thus, analytic curves are coherent and real analytic surface have only isolated points where they fail to be coherent. As coherence is an open condition, $N(X)$ is a closed set. Later Tancredi-Tognoli provided in \cite{tt} a simpler proof of Galbiati's result. Their procedure has helped us to understand the global structure of the set of points of non-coherence of a $C$-analytic set and we describe it carefully in Theorem \ref{ncp0}. As a consequence, we have:

\begin{cor}\label{ncc}
The set $N(X)$ of points of non-coherence of a $C$-analytic set $X\subset M$ is a $C$-semianalytic set of dimension $\leq\dim(X)-2$.
\end{cor}

\subsubsection{Images of $C$-semianalytic sets under proper analytic maps and subanalytic sets}

The family of semianalytic sets is not closed under the image of proper analytic maps. The concept of subanalytic set was introduced to get rid of this problem. Recall that $S\subset M$ is \em subanalytic \em if each point $x\in M$ admits a neighborhood $U^x$ such that $S\cap U^x$ is a projection of a relatively compact semianalytic set, that is, there exists a real analytic manifold $N$ and a relatively compact semianalytic subset $A$ of $M\times N$ such that $S\cap U^x=\pi(A)$ where $\pi:M\times N\to M$ is the projection. 

It could sound reasonable to consider the family of $C$-subanalytic sets. However, this is useless because each subanalytic set is the image of a $C$-semianalytic set under a proper analytic map.

\begin{thm}\label{sub}
Let $S$ be a subset of a real analytic manifold $N$. The following assertions are equivalent: 
\begin{itemize}
\item[(i)] $S$ is subanalytic.
\item[(ii)] There exists a basic $C$-semianalytic subset $T$ of a real analytic manifold $M$ and an analytic map $f:M\to N$ such that $f|_{\cl(T)}:\cl(T)\to N$ is proper and $S=f(T)$.
\item[(iii)] There exists a $C$-semianalytic subset $T$ of a real analytic manifold $M$ and an analytic map $f:M\to N$ such that $f|_{\cl(T)}:\cl(T)\to N$ is proper and $S=f(T)$.
\end{itemize}
\end{thm} 

The family of subanalytic sets is closed under the following operations: locally finite unions and intersections, complement \cite{ga}, closure, interior, connected components, sets of points of dimension $k$, inverse images of analytic maps and direct images of proper analytic maps. As before, this family is not closed under the image of general analytic maps (see Example \ref{notsub}).

\subsection{Structure of the article}

The article is organized as follows. In Section \ref{s2} we present all basic concepts and notations used in this paper as well as some preliminary results concerning: complexification, points of non-coherence, normalization, excellence of localizations of rings of holomorphic functions, etc. that will be useful in the subsequent sections. The reading can be started directly in Section \ref{s3} and referred to the Preliminaries only when needed. The aim of Section \ref{s3} is to study the main properties of $C$-semianalytic sets. In Section \ref{s4} we prove Theorem \ref{properint} and Theorem \ref{finite76} while in Section \ref{s5} we study the structure of the set of points of non-coherence of a $C$-analytic set and we conclude that it is a $C$-semianalytic set (Corollary \ref{ncc}). Finally, in Section \ref{s6} we show Theorem \ref{sub} and deduce that each subanalytic set is the image under a proper analytic map of a basic $C$-semianalytic set.

\section{Preliminaries on real and complex analytic spaces}\label{s2}

In the following \em holomorphic \em refer to the complex case and \em analytic \em to the real case. For a further reading about complex analytic spaces we refer to \cite{gr} while we remit the reader to \cite{gmt,t} for the theory of real analytic spaces. We denote the elements of $\an(X):=H^0(X,\an_X)$ with capital letters if $(X,\an_X)$ is a Stein space and with small letters if $(X,\an_X)$ is a real analytic space. All concepts that appear in the Introduction involving a real analytic manifold $M$ can be extended to a real analytic space $(X,\an_X)$ using the ring $\an(X)$ of global analytic sections on $X$ instead of the ring $\an(M)$ of global analytic functions on $M$. 

\subsection{General terminology}
Denote the coordinates in $\C^n$ with $z:=(z_1,\ldots,z_n)$ where $z_i:=x_i+\sqrt{-1}y_i$. Consider the conjugation $\ol{\,\cdot\,}:\C^n\to\C^n,\ z\mapsto\ol{z}:=(\ol{z_1},\ldots,\ol{z_n})$ of $\C^n$, whose set of fixed points is $\R^n$. A subset $A\subset\C^n$ is \em invariant \em if $\ol{A}=A$. Obviously, $A\cap\ol{A}$ is the biggest invariant subset of $A$. Let $\Omega\subset\C^n$ be an invariant open set and $F:\Omega\to\C$ a holomorphic function. We say that $F$ is \em invariant \em if $F(z)=\ol{F(\ol{z})}$ for all $z\in\Omega$. This implies that $F$ restricts to a real analytic function on $\Omega\cap\R^n$. Conversely, if $f$ is analytic on $\R^n$, it can be extended to an invariant holomorphic function $F$ on some invariant open neighborhood $\Omega$ of $\R^n$. If $F:\Omega\to\C$ is a holomorphic function and $\Omega$ is invariant,
$$
\Re(F):\Omega\to\C,\ z\mapsto\tfrac{F(z)+\ol{F(\ol{z})}}{2}\quad\text{and}\quad\Im(F):\Omega\to\C,\ z\mapsto\tfrac{F(z)-\ol{F(\ol{z})}}{2\sqrt{-1}}
$$ 
are invariant holomorphic functions that satisfy $F=\Re(F)+\sqrt{-1}\,\Im(F)$. 

\subsection{Reduced analytic spaces \cite[I.1]{gmt}}
Let $\K=\R$ or $\C$ and let $(X,\an_X)$ be an either complex or real analytic space. Let ${\mathcal F}_X$ be the sheaf of $\K$-valued functions on $X$ and let $\vartheta:\an_X\to{\mathcal F}_X$ be the morphism of sheaves defined for each open set $U\subset X$ by $\vartheta_U(s):U\to\K,\ x\mapsto s(x)$ where $s(x)$ is the class of $s$ module the maximal ideal $\gtm_{X,x}$ of $\an_{X,x}$. Recall that $(X,\an_X)$ is \em reduced \em if $\vartheta$ is injective. Denote the image of $\an_X$ under $\vartheta$ with $\an_X^r$. The pair $(X,\an_X^r)$ is called the \em reduction \em of $(X,\an_X)$ and $(X,\an_X)$ is reduced if and only if $\an_X=\an_X^r$. The reduction is a covariant functor from the category of $\K$-analytic spaces to that of reduced $\K$-analytic spaces.

\subsection{Anti-involution and complexifications \cite[II.4]{gmt}}
Let $(X,\an_X)$ be a complex analytic space. An \em anti-involution \em on $(X,\an_X)$ is a morphism $\sigma:(X,\an_X)\to(X,\ol{\an}_X)$ such that $\sigma^2=\id$ and it transforms the sheaf of holomorphic sections $\an_X$ into the sheaf of antiholomorphic sections $\ol{\an}_X$.

\subsubsection{Fixed part space}\label{fixed}
Let $(X,\an_X)$ be a complex analytic space endowed with an anti-involu\-tion $\sigma$. Let $X^\sigma:=\{x\in X:\ \sigma(x)=x\}$ and define a sheaf $\an_{X^\sigma}$ on $X^\sigma$ in the following way: for each open subset $U\subset X^\sigma$, we set $H^0(U,\an_{X^\sigma})$ as the subset of $H^0(U,\an_X|_{X^\sigma})$ of invariant sections. The $\R$-ringed space $(X^\sigma,\an_{X^\sigma})$ is called the \em fixed part space of $(X,\an_X)$ with respect to $\sigma$\em. By \cite[II.4.10]{gmt} it holds that $(X^\sigma,\an_{X^\sigma})$ is a real analytic space if $X^\sigma\neq\varnothing$.

\subsubsection{Complexification and $C$-analytic spaces \cite[III.3]{gmt}}\label{complexification} 
A real analytic space $(X,\an_X)$ is a \em $C$-analytic space \em if it satisfies one of the following two equivalent conditions:
\begin{itemize}
\item[(1)] Each local model of $(X,\an_X)$ is defined by a coherent sheaf of ideals, which is not necessarily associated to a \em well reduced structure \em (see \ref{cas}).
\item[(2)] There exist a complex analytic space $(\widetilde{X},\an_{\widetilde{X}})$ endowed with an anti-holomorphic involution $\sigma$ whose fixed part space is $(X,\an_X)$.
\end{itemize}
The complex analytic space $(\widetilde{X},\an_{\widetilde{X}})$ is called a \em complexification \em of $X$ and it satisfies the following properties:
\begin{itemize}
\item[(i)] $\an_{\widetilde{X},x}=\an_{X,x}\otimes\C$ for all $x\in X$.
\item[(ii)] The germ of $(\widetilde{X},\an_{\widetilde{X}})$ at $X$ is unique up to an isomorphism.
\item[(iii)] $X$ has a fundamental system of invariant open Stein neighborhoods in $\widetilde{X}$. 
\item[(iv)] If $X$ is reduced, then $\widetilde{X}$ is also reduced.
\end{itemize} 
For further details see \cite{c,gmt,t,wb}. 

\subsection{$C$-analytic sets}\label{cas} 
The concept of $C$-analytic sets was introduced by Cartan in \cite[\S7,\S10]{c}. Recall that a subset $X\subset M$ is \em $C$-analytic \em if there exists a finite set $\Ss:=\{f_1,\ldots,f_r\}$ of real analytic functions $f_i$ on $M$ such that $X$ is the common zero-set of $\Ss$. This property is equivalent to the following:
\begin{itemize}
\item[(1)] There exists a coherent sheaf of ideals $\Jhaz$ on $M$ such that $X$ is the zero set of $\Jhaz$.
\item[(2)] There exist an open neighborhood $\Omega$ of $M$ in a complexification $\widetilde{M}$ of $M$ and a complex analytic subset $Z$ of $\Omega$ such that $Z\cap M=X$.
\end{itemize}
A coherent analytic set is $C$-analytic, but the converse is not true in general. Consider for example Whitney's umbrella. 

\subsubsection{Well reduced structure}
Given a $C$-analytic set $X\subset M$ the largest coherent sheaf of ideals $\Jhaz$ having $X$ as zero set is $\ideal(X)\an_M$ by Cartan's Theorem $A$, where $\ideal(X)$ is the set of all analytic functions on $M$ that are identically zero on $X$. The coherent sheaf $\an_X:=\an_M/\ideal(X)\an_M$ is called the \em well reduced structure of $X$\em. The $C$-analytic set $X$ endowed with its well reduced structure is a $C$-analytic space, so it has a well-defined complexification as commented above. 

\subsubsection{Singular set of a $C$-analytic set}\label{sing}

Let $X\subset M$ be a $C$-analytic set and let $\widetilde{X}$ be a complexification of $X$. We define the \em singular locus of $X$ \em as $\Sing(X):=\Sing(\widetilde{X})\cap M$. Its complement $\Reg(X):=X\setminus\Sing(X)$ is the set of \em regular points of $X$\em. Observe that $\Sing(X)$ is a $C$-analytic set of strictly smaller dimension than $X$. We define inductively $\Sing_\ell(X)=\Sing(\Sing_{\ell-1}(X))$ for $\ell\geq1$ where $\Sing_1(X)=\Sing(X)$. In particular, $\Sing_\ell(X)=\varnothing$ if $\ell\geq\dim(X)+1$. We will write $X:=\Sing_0(X)$ for simplicity.

We say that $x\in X$ is a \em smooth point \em of $X$ if there exists a neighborhood $U$ of $x$ in $M$ such that $X\cap U$ is analytically diffeomorphic to an open subset of $\R^n$. We denote by ${\rm Smooth}(X)$ the set of smooth points of $X$.

\begin{remark}
It holds that $\Sing(X)$ depends only on $X$ and not in the chosen complexification. It is important to distinguish between regular and smooth points because they are different concepts, although the inclusion $\Reg(X)\subset{\rm Smooth}(X)$ always holds.
\end{remark}

\begin{example}
Let $X:=\{(x^2+y^2)xz-y^4=0\}\subset\R^3$. One can check that the set of regular points of $X$ is $\Reg(X)=X\setminus\{x=0,y=0\}$. However, ${\rm Smooth}(X)=X\setminus\{0\}$. To illustrate this fact pick a point $x=(0,0,a)$ with $a\neq0$ and consider the parameterizations 
$$
\varphi_\veps:\{t>0\}\to X\cap\{\veps z>0\},\ (s,t)\mapsto\veps((s^2+t^2)s^2,(s^2+t^2)st,t^4),
$$ 
for $\veps=\pm1$, whose images cover $X\setminus\{0\}$.
\end{example}

\subsubsection{Irreducible components of a $C$-analytic set}
A $C$-analytic set is \em irreducible \em if it is not the union of two $C$-analytic sets different from itself. In addition, $X$ is irreducible if and only if it admits a fundamental system of invariant irreducible complexifications. Given a $C$-analytic set $X$, there is a unique irredundant (countable) locally finite family of irreducible $C$-analytic sets $\{X_i\}_{i\geq1}$ such that $X=\bigcup_{i\geq1}X_i$. The $C$-analytic sets $X_i$ are called the \em irreducible components of $X$\em. For further details see \cite{wb}.

\subsubsection{Set of points of non-coherence of a $C$-analytic set}
Let $X\subset M$ be a $C$-analytic set. Recall that $X$ is {\em coherent} if the sheaf $\J_X$ of germs of analytic functions vanishing identically on $X$ is an $\an_M$-coherent sheaf of modules. If $M$ is a real analytic manifold (so $\an_M$ is a coherent sheaf of rings), $\J_X$ is \em coherent \em at $x\in M$ if and only if it is of \em finite type \em at $x\in M$, that is, there exists an open neighborhood $U$ of $x$ in $M$ and finitely many sections $f_1,\ldots,f_r\in H^0(U,\J_X)$ such that for each $y\in U$ the germs $f_{1,y},\ldots,f_{r,y}$ generate the stalk $\J_{X,y}$ as an $\an_{X,y}$-module. Otherwise, we say that $X$ is {\em non-coherent at $x\in X$}. 

\begin{remark}
The usual definition of coherence for a sheaf $\Fhaz$ of $\an_M$-modules says that $\Fhaz$ is coherent if and only if $\Fhaz$ is of finite type and of \em finite presentation\em, that is, for every open set $U\subset M$ and every morphism $\varphi:(\an_M|_U)^p\to\Fhaz|_U$ the $\an_M|_U$ module $\ker(\varphi)$ is of finite type. However as $\an_M$ is a coherent sheaf of rings and $\J_X$ is a subsheaf of ideals of $\an_M$, then $\J_X$ is coherent if and only if it is the finite type.
\end{remark}

Recall the following two useful facts concerning coherence:

(i) Let $\widetilde{X}$ be a complex analytic subset of an open subset of $\C^n$. Let $x\in X$ be such that $\widetilde{X}_x$ is the complexification of $X_x$. Then $X$ is coherent at $x$ if and only if there exists an open neighborhood $U$ of $x$ in $M$ such that for each $y\in X\cap U$ the complexification of $X_y$ is $\widetilde{X}_y$ (see \cite[III.2.8]{gmt} or \cite[V.\S1.Prop.5, pag.94]{n}). In particular, irreducible coherent analytic set germs are pure dimensional.

(ii) $X$ is coherent at a point $x\in X$ if and only if the irreducible components of the germ $X_x$ are coherent (see \cite[III.2.13]{gmt} or \cite[V.\S1.Prop.6, pag.94]{n}).

\subsection{Normalization of complex analytic spaces}\label{norm}
One defines the normalization of a complex analytic space in the following way \cite[VI.2]{n}. A complex analytic space $(X,\an_X)$ is \em normal \em if for all $x\in X$ the local analytic ring $\an_{X,x}$ is reduced and integrally closed. A \em normalization \em $(Y,\pi)$ of a complex analytic space $(X,\an_X)$ is a normal complex analytic space $(Y,\an_Y)$ together with a proper surjective holomorphic map $\pi:Y\to X$ with finite fibers such that $Y\setminus\pi^{-1}(\Sing(X))$ is dense in $Y$ and $\pi|:Y\setminus\pi^{-1}(\Sing(X))\to X\setminus\Sing(X)$ is an analytic isomorphism. The normalization $(Y,\pi)$ of a reduced complex analytic space $X$ always exists and is unique up to isomorphism \cite[VI.2.Lem.2 \& VI.3.Thm.4]{n}. 

The following example is inspired in one already proposed in \cite[Esempio, p. 211]{t} and it shows that the concepts of normality of the complexification and coherence are independent.

\begin{example}\label{ex1}
Consider the $C$-analytic set $X:=\{w^2-z(x^2+y^2)=0\}\subset\R^4$ and its complexification $\widetilde{X}=\{w^2-z(x^2+y^2)=0\}\subset\C^4$. The set of singular points of $\widetilde{X}$ is the complex analytic set $\Sing(\widetilde{X})=\{x=0,y=0,w=0\}\cup\{x^2+y^2=0,z=0,w=0\}\subset\C^4$, which has codimension $2$ in $\widetilde{X}$. As $\widetilde{X}$ is a complex irreducible analytic hypersurface, we deduce by \cite{ok} that $\widetilde{X}$ is a normal complex analytic set. Thus, $X$ is a normal $C$-analytic set. As the points of $X$ satisfy the equation $w^2=z(x^2+y^2)$, the germs $X_{p}$ at the points $p:=(0,0,z,0)\in X$ with $z<0$ have dimension $1$. Consequently, $X$ is not pure dimensional and consequently non-coherent. 
\end{example}

\subsection{Localizations and excellent rings}

Let $(X,\an_X)$ be a reduced Stein space endowed with an anti-involution $\sigma$ such that its fixed part space $X^\sigma$ is non-empty. Let ${\mathcal A}(X)\subset H^0(X,\an_X)$ be the subring of all invariant holomorphic sections of $H^0(X,\an_X)$ and 
$$
{\mathcal A}(X^\sigma):=\{F|_{X^\sigma}:\ F\in{\mathcal A}(X)\}\subset\an(X^\sigma). 
$$
Let $x_0\in X^\sigma$ and denote the respective maximal ideals of $\an(X^\sigma)$ and ${\mathcal A}(X^\sigma)$ associated to $x_0$ with $\gtm_{x_0}$ and $\gtn_{x_0}$. We refer the reader to \cite{m} for the theory of excellent rings and regular homomorphisms. We denote the completion of a local ring $A$ with $\widehat{A}$.

\begin{lem}\label{excell}
Let $x_0\in\R^n$. Then the rings $\an(\R^n)_{\gtm_{x_0}}$ and ${\mathcal A}(\R^n)_{\gtn_{x_0}}$ are excellent rings, 
$$
\widehat{\an(\R^n)_{\gtm_{x_0}}}\cong\widehat{\an_{\R^n,x_0}}\cong\widehat{{\mathcal A}(\R^n)_{\gtn_{x_0}}}
$$ 
and the homomorphisms 
$$
\an(\R^n)_{\gtm_{x_0}}\to\an_{\R^n,x_0},\ \xi\mapsto\xi_{x_0}\quad\text{and}\quad
{\mathcal A}(\R^n)_{\gtn_{x_0}}\to\an_{\R^n,x_0},\ \zeta\mapsto\zeta_{x_0}
$$
are regular.
\end{lem}
\begin{proof}
In \cite[VIII.4.4(a,b)]{abr} it is proved the statement for $\an(\R^n)_{\gtm_{x_0}}$, so we focus on ${\mathcal A}(\R^n)_{\gtn_{x_0}}$. The restriction homomorphism 
$$
{\mathcal A}(\C^n)\to{\mathcal A}(\R^n), F\mapsto F|_{\R^n}
$$
is an isomorphism, so we identify both rings and prove the results for ${\mathcal A}(\C^n)_{\gtn_{x_0}}$. Denote the maximal ideal of $\an(\C^n)$ associated to $x_0$ with $\gtg_{x_0}$. Proceeding as in the proof of \cite[VIII.4.4(a,b)]{abr} one shows that $\an(\C^n)_{\gtg_{x_0}}$ is a regular local ring of dimension $n$, the homomorphism $\an(\C^n)_{\gtg_{x_0}}\to\an_{\C^n,x_0}$ is faithfully flat and it extends to an isomorphism between completions. 

\paragraph{}We claim: \em ${\mathcal A}(\C^n)_{\gtn_{x_0}}$ is a regular local ring of dimension $n$ and the homomorphism ${\mathcal A}(\C^n)_{\gtn_{x_0}}\to\an_{\R^n,x_0}$ is faithfully flat and it extends to an isomorphism between completions\em. 

We prove first that ${\mathcal A}(\C^n)_{\gtn_{x_0}}$ is noetherian. Pick an ideal $\gta$ of ${\mathcal A}(\C^n)_{\gtm_{x_0}}$ and consider the extended ideal $\gta\an(\C^n)_{\gtg_{x_0}}$. As $\an(\C^n)_{\gtg_{x_0}}$ is noetherian, there exists finitely many $\zeta_1,\ldots,\zeta_r\in\gta$ that generate $\gta\an(\C^n)_{\gtg_{x_0}}$. Pick an element $\zeta\in\gta$ and write $\zeta=\zeta_1\alpha_1+\cdots+\zeta_r\alpha_r$ for some $\alpha_i\in\an(\C^n)_{\gtg_{x_0}}$. Then
$$
\zeta=\Re(\zeta)=\zeta_1\Re(\alpha_1)+\cdots+\zeta_r\Re(\alpha_r)\in(\zeta_1,\ldots,\zeta_r){\mathcal A}(\C^n)_{\gtm_{x_0}}, 
$$
so $\gta$ is finitely generated and ${\mathcal A}(\C^n)_{\gtm_{x_0}}$ is noetherian. Similarly one shows that the homomorphism ${\mathcal A}(\C^n)_{\gtm_{x_0}}\to\an_{\R^n,x_0}$ is faithfully flat and it extends to an isomorphism between completions (follow the proof of \cite[VIII.4.4(a,b)]{abr} and use the previous trick involving $\Re(\cdot)$ where needed). Consequently, ${\mathcal A}(\C^n)_{\gtm_{x_0}}$ is a regular local ring of dimension $n$.

Consider the partial derivatives $\frac{\partial}{\partial x_i}$ and the projections $\pi_i:\C^n\to\C,\ (x_1,\ldots,x_n)\mapsto x_i$, which belong to ${\mathcal A}(\C^n)_{\gtm_{x_0}}$. It holds $\frac{\partial}{\partial x_i}\pi_j=\delta_{ij}$ for $1\leq i,j\leq n$. By \cite[40.F, Th.102, p.291]{m} we conclude that ${\mathcal A}(\C^n)_{\gtm_{x_0}}$ is an excellent ring.

\paragraph{}Consider the composition of homomorphisms 
$$
{\mathcal A}(\R^n)_{\gtm_{x_0}}\overset{\phi}{\hookrightarrow}\an_{\R^n,x_0}\overset{\psi}{\hookrightarrow}\widehat{\an_{\R^n,x_0}}\cong\widehat{{\mathcal A}(\R^n)_{\gtm_{x_0}}}.
$$
As ${\mathcal A}(\R^n)_{\gtm_{x_0}}$ and $\an_{\R^n,x_0}$ are local excellent rings, the homomorphisms $\psi\circ\phi$ and $\psi$ are regular. As all rings in the row are local, both homomorphisms are faithfully flat. By \cite [33.B, Lemma 1, p.250]{m} $\phi$ is also regular.
\end{proof}

\begin{cor}\label{excellpol}
Let $\u:=(\u_1,\ldots,\u_m)$ be a tuple of variables and fix $x_0\in\R^n$ and $u_0\in\R^m$. Denote the respective maximal ideals of $\an(\R^n)_{\gtm_{x_0}}[\u]$ and ${\mathcal A}(\R^n)_{\gtn_{x_0}}[\u]$ associated to $(x_0,u_0)$ with $\gtm_{(x_0,u_0)}$ and $\gtn_{(x_0,u_0)}$. Then the rings $(\an(\R^n)_{\gtm_{x_0}}[\u])_{\gtm_{(x_0,u_0)}}$ and $({\mathcal A}(\R^n)_{\gtn_{x_0}}[\u])_{\gtn_{(x_0,u_0)}}$ are excellent, 
$$
\widehat{(\an(\R^n)_{\gtm_{x_0}}[\u])_{\gtm_{(x_0,u_0)}}}\cong\widehat{\an_{\R^n\times\R^m,(x_0,u_0)}}\cong\widehat{({\mathcal A}(\R^n)_{\gtn_{x_0}}[\u])_{\gtn_{(x_0,u_0)}}}
$$
and the homomorphisms
\begin{align*}
&(\an(\R^n)_{\gtm_{x_0}}[\u])_{\gtm_{(x_0,u_0)}}\to\an_{\R^n\times\R^m,(x_0,u_0)},\ \xi\mapsto \xi_{(x_0,u_0)}\\
&({\mathcal A}(\R^n)_{\gtn_{x_0}}[\u])_{\gtn_{(x_0,u_0)}}\to\an_{\R^n\times\R^m,(x_0,u_0)},\ \zeta\mapsto \zeta_{(x_0,u_0)}
\end{align*}
are regular.
\end{cor}
\begin{proof}
The rings $(\an(\R^n)_{\gtm_{x_0}}[\u])_{\gtm_{(x_0,u_0)}}$ and $({\mathcal A}(\R^n)_{\gtn_{x_0}}[\u])_{\gtm_{(x_0,u_0)}}$ are excellent by Lemma \ref{excell} and \cite[33.G, Th.77, p.254 \& 32.B, Th.73, p.246]{m}. Consider the compositions of homomorphisms 
\begin{align*}
(\an(\R^n)_{\gtm_{x_0}}[\u])_{\gtm_{(x_0,u_0)}}\overset{\phi_1}{\hookrightarrow}\an_{\R^n\times\R^m,(x_0,u_0)}\overset{\psi}{\hookrightarrow}\widehat{\an_{\R^n\times\R^m,(x_0,u_0)}}\cong\widehat{(\an(\R^n)_{\gtn_{x_0}}[\u])_{\gtm_{(x_0,u_0)}}},\\
({\mathcal A}(\R^n)_{\gtn_{x_0}}[\u])_{\gtn_{(x_0,u_0)}}\overset{\phi_2}{\hookrightarrow}\an_{\R^n\times\R^m,(x_0,u_0)}\overset{\psi}{\hookrightarrow}\widehat{\an_{\R^n\times\R^m,(x_0,u_0)}}\cong\widehat{({\mathcal A}(\R^n)_{\gtn_{x_0}}[\u])_{\gtn_{(x_0,u_0)}}}.
\end{align*}
As $(\an(\R^n)_{\gtm_{x_0}}[\u])_{\gtm_{(x_0,u_0)}}$, $({\mathcal A}(\R^n)_{\gtn_{x_0}}[\u])_{\gtm_{(x_0,u_0)}}$ and $\an_{\R^n\times\R^m,(x_0,u_0)}$ are local excellent rings, the homomorphisms $\psi\circ\phi_1$, $\psi\circ\phi_2$ and $\psi$ are regular. As all rings in the rows are local, the previous homomorphisms are faithfully flat. By \cite [33.B, Lemma 1, p.250]{m} $\phi_1$ and $\phi_2$ are also regular.
\end{proof}

\subsection{Transferring results from complex analytic sets to reduced Stein spaces} 
The following results allows us to reduce `local' problems on a reduced Stein space to the corresponding `local' problems on a complex analytic subset of $\C^n$.

\begin{lem}\label{excel-conv}
Let $(X,\an_X)$ be a reduced Stein space endowed with an anti-involution $\sigma$ such that its fixed part space $X^\sigma$ is non-empty and let $x_0\in X^\sigma$. Then there exist an invariant injective proper holomorphic map $\varphi:X\to\C^n$ with image $Z:=\varphi(X)$ satisfying:
\begin{itemize}
\item[(i)] For each $x\in\Reg(X)\cup\{x_0\}$ there exists an open neighborhood $U$ in $X$ such that $\varphi|_U:U\to\varphi(U)$ is an analytic diffeomorphism.
\item[(ii)] $\varphi^*:\an(Z)_{\gtm_{z_0}'}\to\an(X)_{\gtm_{x_0}},\ \tfrac{F}{G}\mapsto\tfrac{F\circ\varphi}{G\circ\varphi}$, where $z_0:=\varphi(x_0)$, is an isomorphism.
\item[(iii)] $\varphi^*({\mathcal A}(Z)_{\gtn_{z_0}'})={\mathcal A}(X)_{\gtn_{x_0}}$.
\end{itemize}
\end{lem}
\begin{proof}
By \cite[V.3.7-8]{gmt} there exists an invariant injective proper holomorphic map $\varphi:X\to\C^n$ where $n$ is large enough satisfying: 
\begin{itemize}
\item[$\bullet$] For each $x\in\Reg(X)\cup\{x_0\}$ there exists an open neighborhood $U$ in $X$ such that $\varphi|_U:U\to\varphi(U)$ is an analytic diffeomorphism (recall that $\varphi(U)$ is a complex analytic subset of an open neighborhood of $\varphi(x)$ in $\C^n$ by Remmert's Theorem).
\item[$\bullet$] $\varphi(\sigma(x))=\ol{\varphi(x)}$ for all $x\in X$.
\end{itemize}
By Remmert's Theorem $Z$ is a complex analytic subset of $\C^n$. We have $\varphi(X^\sigma)=Z^{\ol{\cdot}}$ is the fixed part space of $Z$ under usual conjugation. 

Let $\gtg_{z_0}'$ be either the zero ideal if $Z$ is irreducible or the maximal ideal of $\mer(Z)$ associated to $z_0$ otherwise. We define analogously $\gtg_{x_0}$ in $\mer(X)$. By \cite[Lem.3.8]{as} the homomorphism
$$
\varphi^*:\mer(Z)\to\mer(X),\ \tfrac{F}{G}\mapsto\tfrac{F\circ\varphi}{G\circ\varphi}
$$
is an isomorphism. The previous isomorphism extends to an isomorphism
$$
{\varphi^*}':\mer(Z)_{\gtg_{z_0}'}\to\mer(X)_{\gtg_{x_0}'},\ \tfrac{F}{G}\mapsto\tfrac{F\circ\varphi}{G\circ\varphi}.
$$ 
Observe that 
\begin{align*}
&{\mathcal A}(Z)_{\gtn_{z_0}'}\subset\an(Z)_{\gtm_{z_0}'}\subset\mer(Z)_{\gtg_{z_0}'},\\
& {\mathcal A}(X)_{\gtn_{x_0}}\subset\an(X)_{\gtm_{x_0}}\subset\mer(X)_{\gtg_{x_0}}.
\end{align*}
Now, one can check ${\varphi^*}'({\mathcal A}(Z)_{\gtn_{z_0}'})={\mathcal A}(X)_{\gtn_{x_0}}$ and ${\varphi^*}'(\an(Z)_{\gtm_{z_0}'})=\an(X)_{\gtm_{x_0}}$. Consequently, statements (ii) and (iii) holds, as required.
\end{proof}

\begin{cor}
Let $(X,\an_X)$ be a reduced Stein space endowed with an anti-involution $\sigma$ such that its fixed part space $X^\sigma$ is non-empty and let $x_0\in X^\sigma$. Then the rings $\an(X^\sigma)_{\gtm_{x_0}}$ and ${\mathcal A}(X^\sigma)_{\gtn_{x_0}}$ are excellent rings and the homomorphisms 
$$
\an(X^\sigma)_{\gtm_{x_0}}\to\an_{X^\sigma,x_0},\ \xi\mapsto\xi_{x_0}\quad\text{and}\quad
{\mathcal A}(X^\sigma)_{\gtn_{x_0}}\to\an_{X^\sigma,x_0},\ \zeta\mapsto\zeta_{x_0}
$$
are regular.
\end{cor}
\begin{proof}
By Lemma \ref{excel-conv} we may assume $X\subset\C^n$ is a complex analytic set. As $\an(X^\sigma)=\an(\R^n)/\ideal(X^\sigma)$ and ${\mathcal A}(X^\sigma)={\mathcal A}(\R^n)/\ideal'(X^\sigma)$ where $\ideal(X^\sigma)$ and $\ideal'(X^\sigma)$ are the corresponding zero ideals of $X$, the statement follows from Lemma \ref{excell}. 
\end{proof}

\begin{lem}\label{transfer}
Let $X,Y$ be two topological spaces and let $\varphi:X\to Y$ be a continuous map. Let $\{x_1,\ldots,x_\ell\}\subset X$ and denote $y_k:=\varphi(x_k)$. Assume that there exists a neighborhood $V$ of $\{y_1,\ldots,y_\ell\}$ such that $\varphi|_U:U:=\varphi^{-1}(V)\to V$ is a homeomorphism. Let ${\mathcal R}(X)$ and ${\mathcal R}(Y)$ be two rings of continuous real-valued functions. Denote 
\begin{itemize}
\item[(i)] the multiplicative set of the functions of ${\mathcal R}(X)$ that do not vanish at each $x_k$ with $\Ss$, 
\item[(ii)] the multiplicative set of the functions of ${\mathcal R}(Y)$ that do not vanish at each $y_k$ with $\Tt$.
\end{itemize}
Assume that $\varphi^*:\Tt^{-1}{\mathcal R}(Y)\to\Ss^{-1}{\mathcal R}(X),\ \xi\mapsto\xi\circ\varphi$ is an isomorphism. Let $f_{ij}\in{\mathcal R}(X)$ for $1\leq i\leq r$ and $1\leq j\leq s$ and define $S:=\bigcup_{i=1}^r\bigcap_{j=1}^s\{f_{ij}*_{ij}0\}$ where $*_{ij}$ is either $>$ or $=$. Assume $S\subset U$. Then there exists $h,g_{ij}\in{\mathcal R}(Y)$ such that $h(y_k)\neq0$ and 
\begin{align}
\{(h\circ\varphi)^2>0\}\cap S&=\{(h\circ\varphi)^2>0\}\cap\bigcup_{i=1}^r\bigcap_{j=1}^s\{(g_{ij}\circ\varphi)*_{ij}0\},\label{ok1}\\
\varphi(\{(h\circ\varphi)^2>0\}\cap S)&=\{h^2>0\}\cap\bigcup_{i=1}^r\bigcap_{j=1}^s\{g_{ij}*_{ij}0\}\label{ok2}.
\end{align}
\end{lem}
\begin{proof}
Let $g_{ij}\in{\mathcal R}(Y)$ and $h\in\Tt$ be such that
$$
\varphi^*\Big(\frac{g_{ij}}{h^2}\Big)=f_{ij}\quad\leadsto\quad g_{ij}\circ\varphi=f_{ij}\cdot(h\circ\varphi)^2. 
$$
Note that
$$
\{f_{ij}*_{ij}0\}\cap\{(h\circ\varphi)^2>0\}=\{(g_{ij}\circ\varphi)*_{ij}0\}\cap\{(h\circ\varphi)^2>0\}.
$$
Consequently, equality \eqref{ok1} holds. In addition, as the restriction $\varphi|_U:U\to V$ is a homeomorphism, $\varphi^{-1}(\varphi(U))=U$ and $\{(h\circ\varphi)^2>0\}\cap S\subset U$, equation \eqref{ok2} also holds, as required.
\end{proof}

\section{$C$-semianalytic sets}\label{s3}

A subset $S\subset M$ is a \em basic $C$-semianalytic \em if it admits a description of the type
$$
S:=\{x\in M:\ f(x)=0,g_1(x)>0,\ldots,g_r(x)>0\}
$$
where the functions $f,g_i\in\an(M)$. To uniform notations $S$ is a \em global $C$-semianalytic set \em if it is a finite union of basic $C$-semianalytic sets. We say that $S\subset M$ is a \em $C$-semianalytic set \em if it satisfies one of the following equivalent conditions:
\begin{itemize}
\item[(1)] $S$ is the union of a countable locally finite family of basic $C$-semianalytic sets.
\item[(2)] For each point $x\in M$ there exists an open neighborhood $U^x$ such that $S\cap U^x$ is a global $C$-semianalytic set. 
\end{itemize}

Recall that if $\{X_i\}_{i\geq1}$ is a locally finite family of $C$-semianalytic sets, then $X:=\bigcup_{i\geq1}X_i$ is a $C$-semianalytic set. To show the equivalence of conditions (1) and (2) above we prove first the following result concerning countable locally finite analytic refinements.

\begin{lem}[Countable locally finite analytic refinement]\label{cover}
Let ${\mathscr U}:=\{U_i\}_{i\in I}$ be an open covering of a real analytic manifold $M$. Then there exists a countable locally finite open refinement ${\mathscr V}:=\{V_j\}_{j\geq1}$ of ${\mathscr U}$ such that each $V_j=\{g_j>0\}$ for some analytic function $g_j\in\an(M)$.
\end{lem}
\begin{proof}
As $M$ is paracompact and it admits a countable exhaustion by compact sets, there exist countable locally finite open refinements ${\mathscr W}:=\{W_j\}_{j\geq1}$ and ${\mathscr W}':=\{W_j'\}_{j\geq1}$ of ${\mathscr U}$ such that $\cl(W_j')\subset W_j$ for each $j\geq1$. As the closed sets $\cl(W_j')$ and $M\setminus W_j$ are disjoint, there exists a continuous function $f_j:M\to\R$ such that $f_j|_{\cl(W_j')}\equiv 1$ and $f_j|_{M\setminus W_j}\equiv-1$. Let $g_j$ be an analytic approximation of $f_j$ such that $|g_j-f_j|<\frac{1}{2}$. Notice that
$$
\cl(W_j')\subset V_j:=\{g_j>0\}\subset W_j
$$ 
for each $j\geq1$. Thus, ${\mathscr V}:=\{V_j\}_{j\geq1}$ is a countable locally finite open refinement of ${\mathscr U}$.
\end{proof}

\begin{proof}[Proof of the equivalence of conditions \em (1) \em and \em (2) \em above]
Assume first that $S$ satisfies condition (1) and let $x\in M$. Let $g\in\an(M)$ be such that $x\in U:=\{g>0\}$ and $U$ intersects only finitely many $S_i$, say $S_1,\ldots,S_r$. Thus, $S\cap U=\bigcup_{i=1}^rS_i\cap\{g>0\}$ is global $C$-semianalytic.

Conversely, assume $S$ satisfies condition (2). For each $x\in M$ let $W^x\subset M$ be an open neighborhood of $x$ such that $S\cap W^x$ is a global $C$-semianalytic set. By Lemma \ref{cover} there exists a countable locally finite refinement ${\mathscr V}:=\{V_j\}_{j\geq1}$ of ${\mathscr W}:=\{W^x\}_{x\in M}$ such that each $V_j=\{g_j>0\}$ for some $g_j\in\an(M)$. For each $j\geq1$ let $x_j\in M$ be such that $V_j\subset W^{x_j}$ and let $S'_{j1},\ldots,S'_{j,r_j}$ be basic $C$-semianalytic sets such that $S\cap W^{x_j}=S'_{j1}\cup\cdots\cup S'_{j,r_j}$. Therefore
$$
S\cap V_j=S_{j1}\cup\cdots\cup S_{j,r_j}
$$
where $S_{jk}:=S_{jk}'\cap V_j$ is a basic $C$-semianalytic set. Notice that 
$$
S=S\cap\bigcup_{j\geq1}V_j=\bigcup_{j\geq1}S\cap V_j=\bigcup_{j\geq1}(S_{j1}\cup\ldots\cup S_{j,r_j})
$$
and that the family $\{S_{jk}:\ 1\leq k\leq r_j\}_{j\geq1}$ is locally finite because ${\mathscr V}$ is so. After reindexing the basic $C$-semianalytic sets $S_{jk}$ we are done.
\end{proof}

\begin{example}
Let $S:=\{x^2-(z^2-1)y^2=0, z>-\frac{1}{2}\}\cup\{x=0,y=0\}\subset\R^3$ be the double umbrella suggested by Coste (see \cite[\S1]{q}). 

The smallest $C$-analytic set that contains the $C$-semianalytic set $S$ is 
$$
X:=\{x^2-(z^2-1)y^2=0\}. 
$$
For each $x\in\R^3$ there exists an open neighborhood $U^x$ such that the intersection
$$
S\cap U^x=\{f=0\}\cap U^x
$$ 
for some $f\in\an(\R^3)$. However $S$ is not a $C$-analytic set because $S\subsetneq X$. Thus, if a $C$-semianalytic set admits local descriptions as a $C$-analytic set, it is not guaranteed that $S$ is in addition a $C$-analytic set. 
\end{example}

\begin{center}
\begin{figure}[ht]
\begin{minipage}{0.49 \textwidth} % 0.5 \textwidth = 50% ancho de p\'agina
\begin{center}
\includegraphics[width=.70\textwidth]{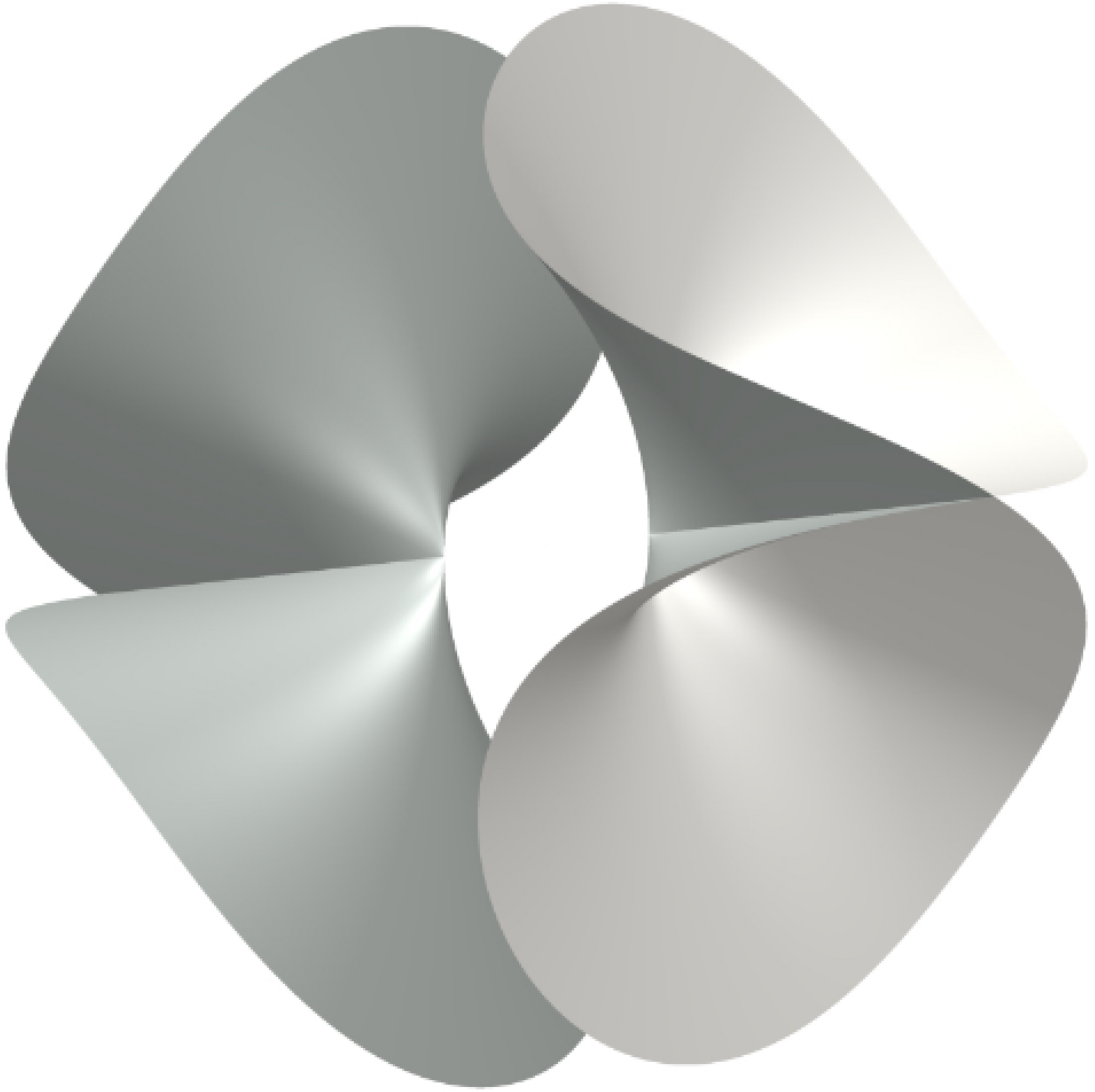}
\end{center}
\end{minipage}
\hfill 
\begin{minipage}{0.50 \textwidth} % 0.5 \textwidth = 50% ancho de p\'agina
\begin{center}
\includegraphics[width=.85\textwidth]{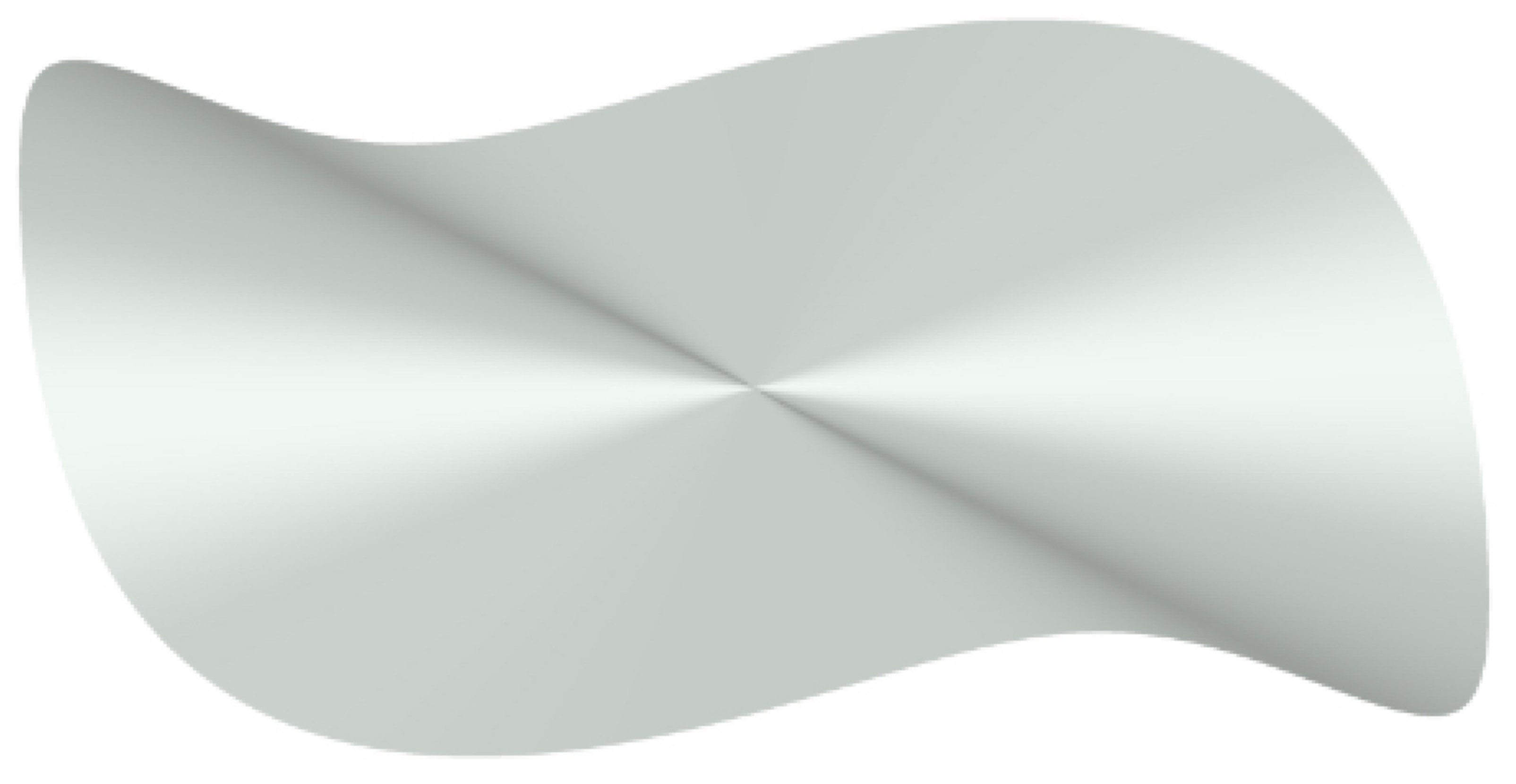}
\end{center}
\end{minipage}
\caption{Two dimensional parts of the double umbrella $x^2-(z^2-1)y^2=0$ and Cartan's umbrella.}
\end{figure}
\end{center}

There exist many semianalytic sets that are not $C$-semianalytic sets. We present next some examples.

\begin{examples}\label{counterexamples}
(i) We recall first a classical example of Cartan \cite[\S11]{c}. Consider the continuous function $f:\R\to\R$ given by the following formula: 
$$
a(z):=\begin{cases}
{\rm exp}(\frac{1}{z^2-1})&\text{if $|z|<1$,}\\
0&\text{if $|z|\geq1$.}
\end{cases}
$$
Clearly, $a$ is analytic on $\R\setminus\{-1,1\}$. Let $Z$ be the closed set of equation $z(x^2+y^2)-x^3a(z)=0$. Observe that $S:=Z\cap U$ is a $C$-analytic subset of $U:=\R^3\setminus\{x=0,y=0,|z|\geq1\}$. On the other hand, 
$$
Z\cap\{x^2+y^2<\veps,|z|>\tfrac{1}{2}\}=\{x=0,y=0,|z|>\tfrac{1}{2}\} 
$$
for $\veps>0$ small enough. Consequently, \em $Z$ is a semianalytic set\em. However, \em it is not a $C$-semianalytic subset of any of its open neighborhoods in $\R^3$\em. 

Fix an open neighborhood $U$ of $Z$ in $\R^3$. By \cite[Prop.18]{c} every analytic function $f$ on $U$ that vanishes at each point of $Z$ is identically zero. Assume that $Z$ is a $C$-semianalytic set. Then there exists an open neighborhood $V$ of the origin such that $Z\cap V$ is a finite union of basic $C$-semianalytic sets. One of these basic $C$-semianalytic sets $$S:=\{f=0,g_1>0,\ldots,g_r>0\}$$ contains a non-empty open subset of the connected real analytic manifold $N:=Z\setminus\{x=0,y=0\}$. As $f$ vanishes identically on an open subset of $N$, it is identically zero on $N$.

As $Z_0$ is an irreducible germ and $\{x=0,y=0\}$ is irreducible, $f$ vanishes identically on the line $\{x=0,y=0\}$. Thus, $f$ is identically zero on $Z$ and consequently on $U$, so 
$$
S=\{f=0,g_1>0,\ldots,g_r>0\}=\{g_1>0,\ldots,g_r>0\}
$$ 
is a non-empty open subset of $U$, a contradiction.

(ii) Consider the closed subset $Z\subset\R^3$ defined by the equation
$$
f:=(1-4(x^2+y^2+z^2))x^2-((y-1)^2+z^2-1)^2a(z)=0
$$
where $a(z)$ is the same function as in the previous example. \em It is a semianalytic set and it is compact\em. Observe that $S$ is the union of the circle $C_1:=\{x=0,(y-1)^2+z^2-1=0\}$ with a two dimensional $C$-semianalytic set $S_{(2)}$ contained in the ball $\{x^2+y^2+z^2\leq\frac{1}{4}\}$. As before, one proves that \em it is not a $C$-semianalytic subset of any of its open neighborhoods in $\R^3$\em.

\begin{center}
\begin{figure}[ht]
\begin{minipage}{0.24 \textwidth} % 0.5 \textwidth = 50% ancho de p\'agina
\begin{center}
\hspace*{-3mm}\includegraphics[width=.60\textwidth]{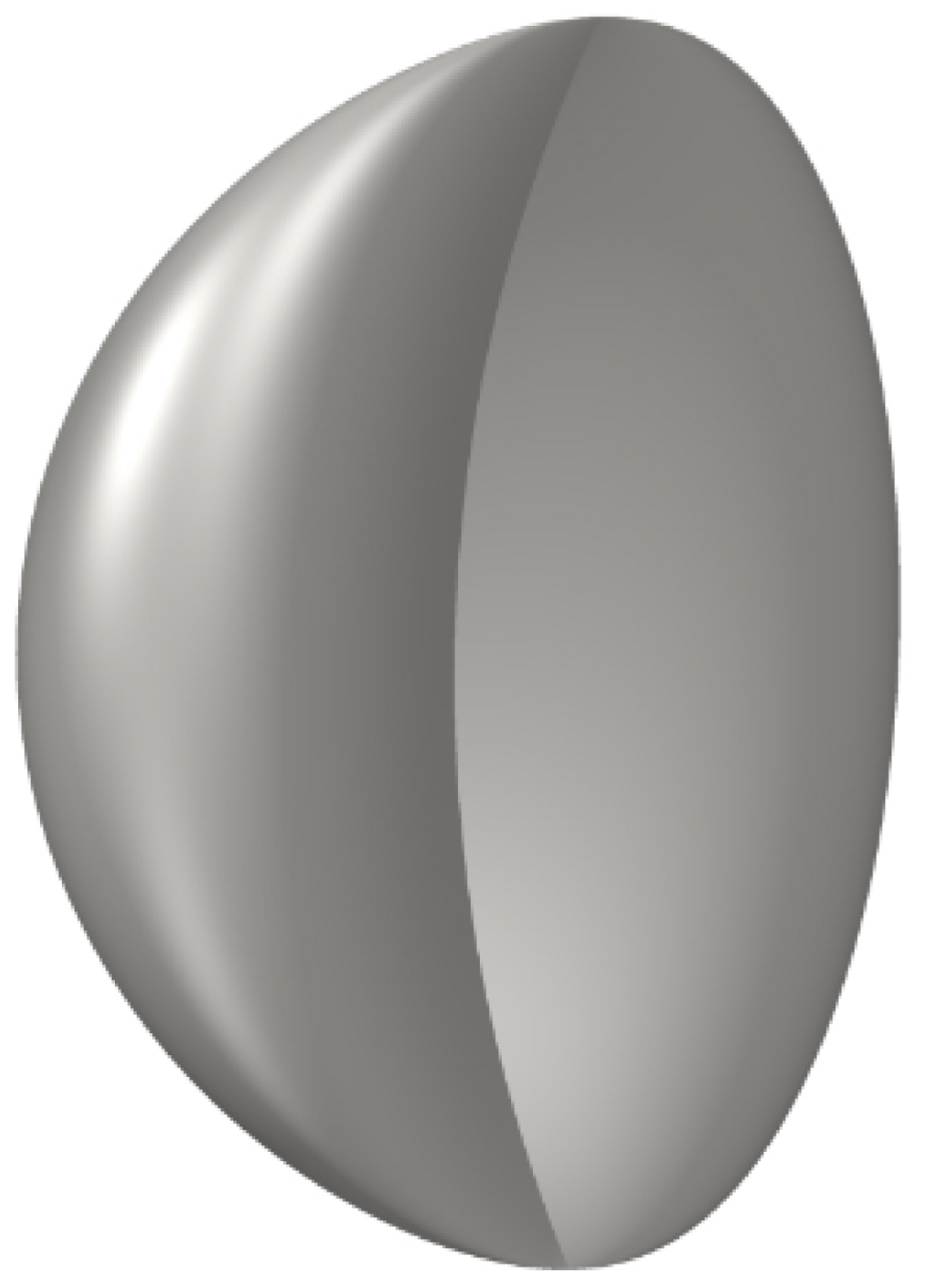}
\end{center}
\end{minipage}
\hfill 
\begin{minipage}{0.24 \textwidth} % 0.5 \textwidth = 50% ancho de p\'agina
\begin{center}
\hspace*{-3mm}\includegraphics[width=.85\textwidth]{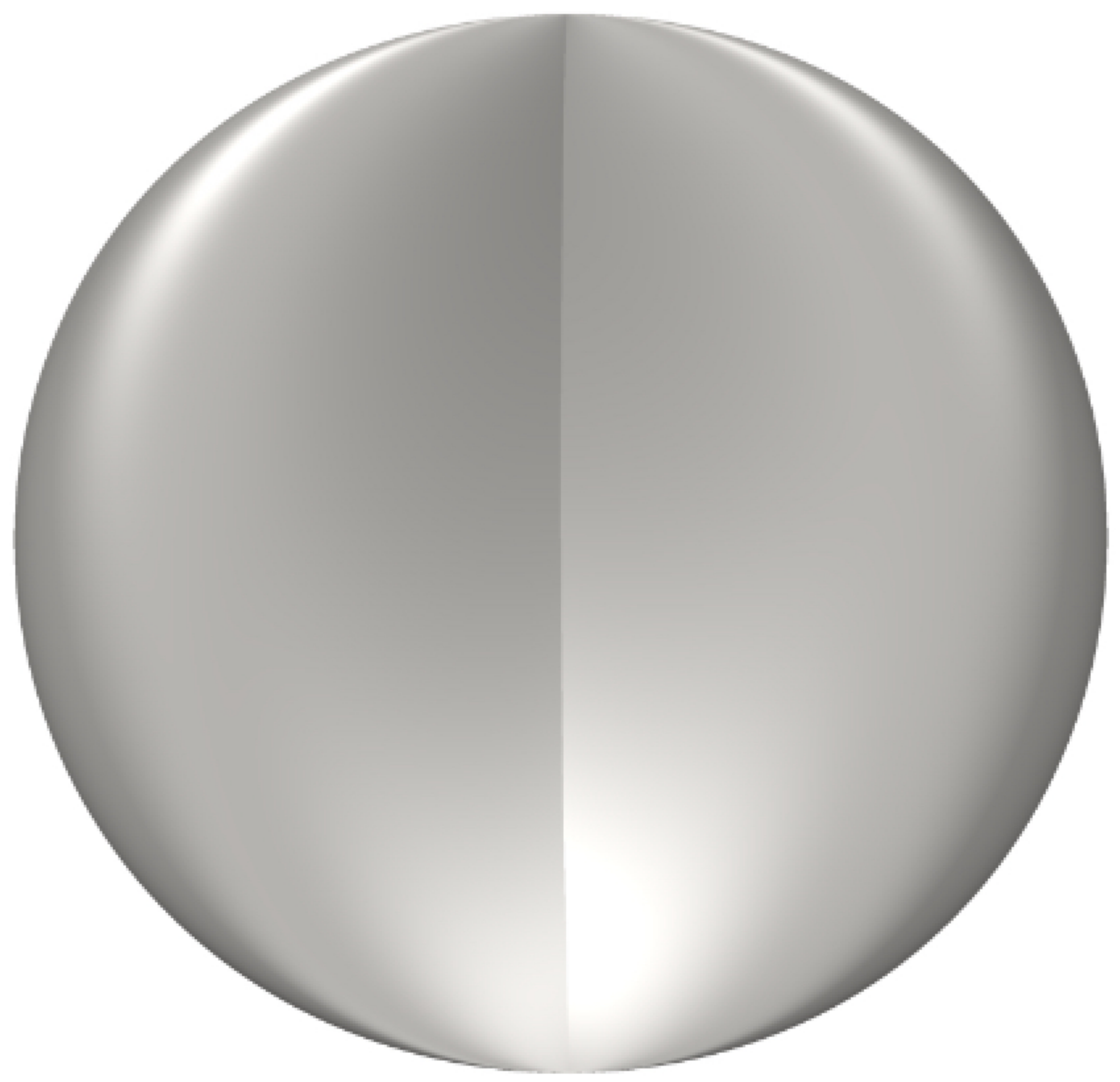}
\end{center}
\end{minipage}
\hfill 
\begin{minipage}{0.24\textwidth}
\begin{center}
\hspace*{-3mm}\includegraphics[width=.85\textwidth]{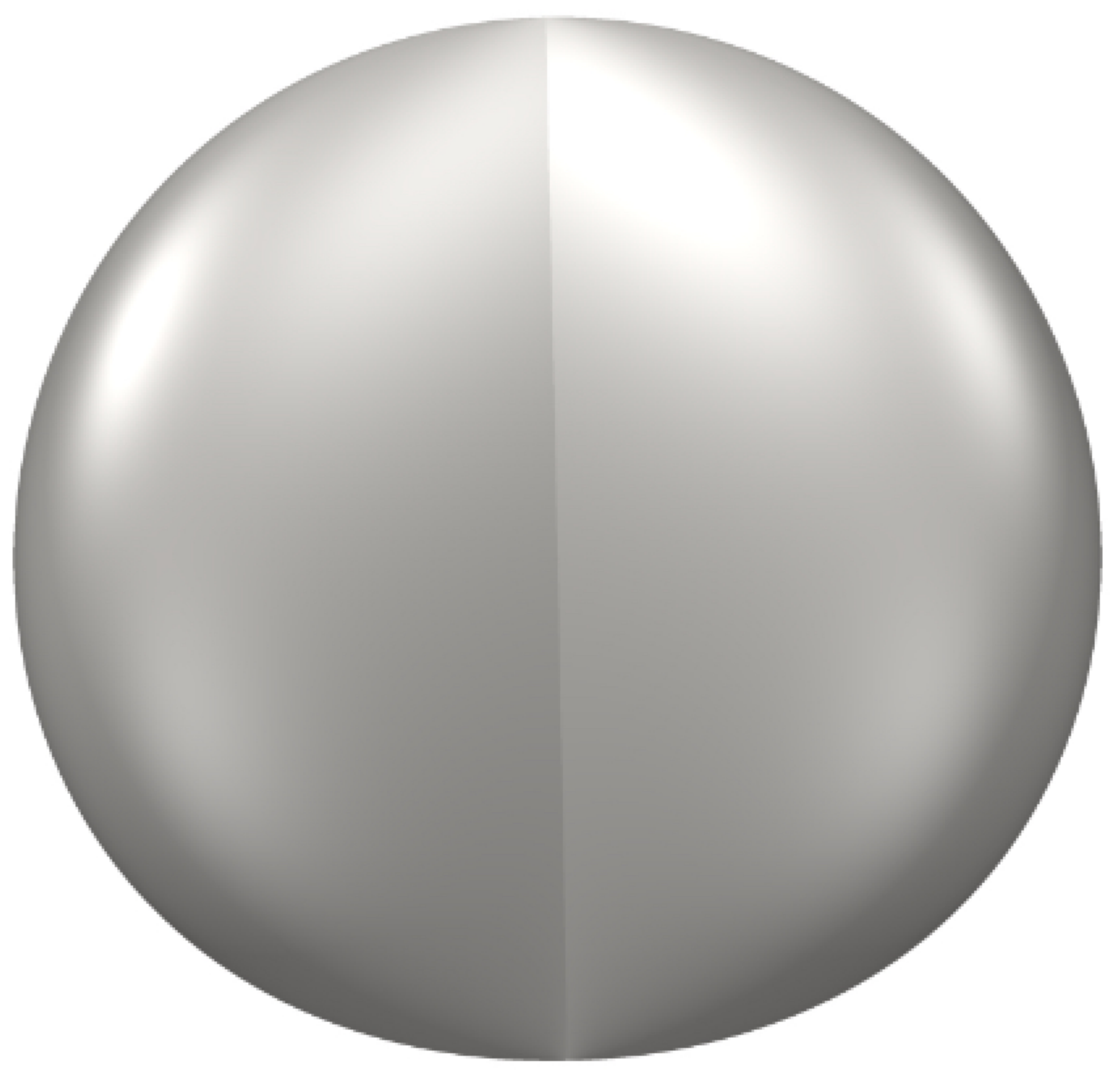}
\end{center}
\end{minipage}
\hfill
\begin{minipage}{0.24\textwidth}
\begin{center}
\hspace*{-3mm}\includegraphics[width=.60\textwidth]{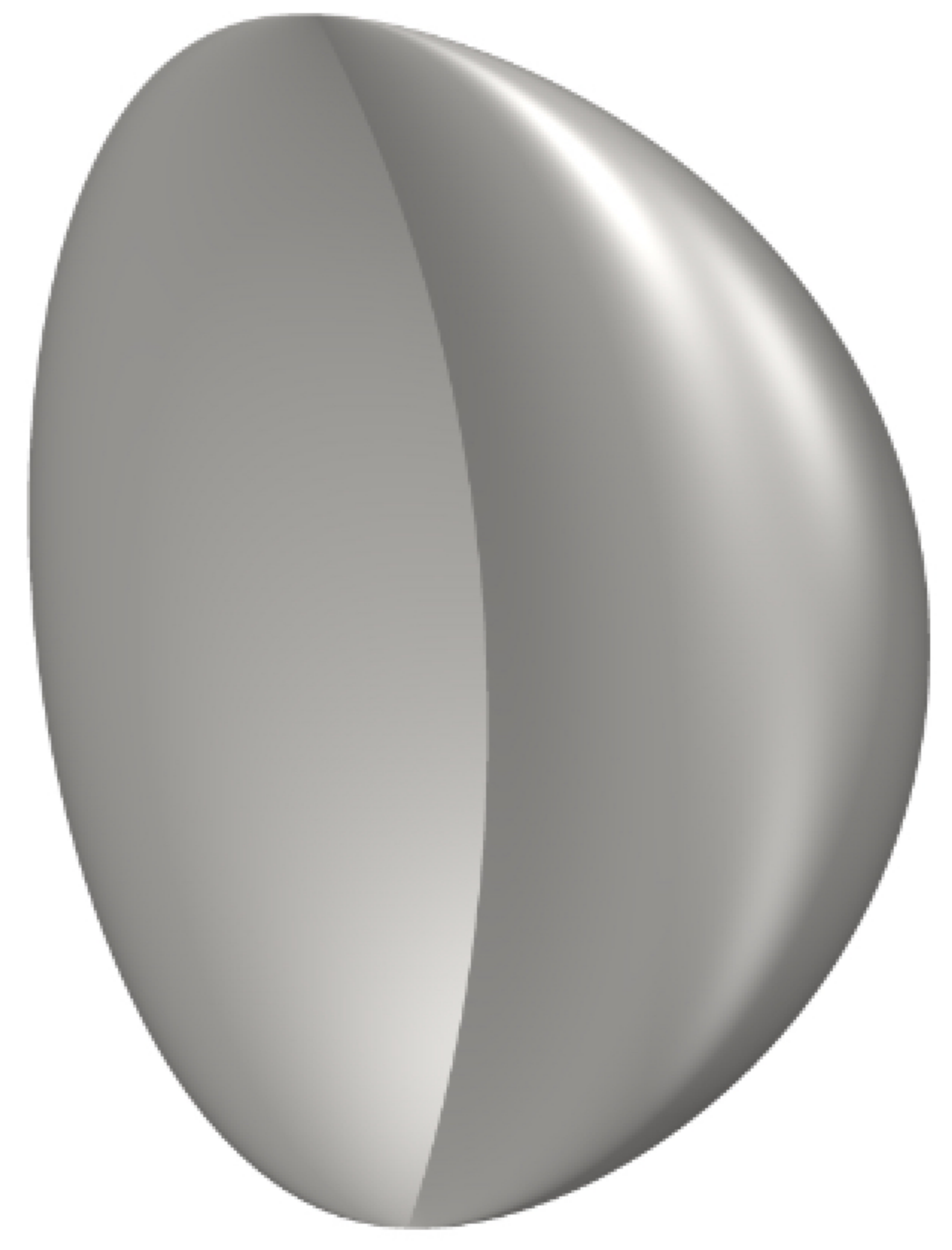}
\end{center}
\end{minipage}

\caption{Several views of the two dimensional part $S_{(2)}$ of the seminanalytic set $S$ in Example \ref{counterexamples}(ii)}
\end{figure}
\end{center}
\end{examples}

\subsection{Basic properties of $C$-semianalytic sets}\label{bpcss}
The family of $C$-semianalytic sets is closed under: 
\begin{itemize}
\item locally finite unions, 
\item locally finite intersections, 
\item complement, 
\item inverse image under analytic maps between real analytic manifolds, 
\item taking closure and interior and 
\item considering connected components.
\end{itemize}
\begin{proof}
The first three properties are clear. For the remaining ones we proceed as follows.

\paragraph{}\em Inverse image\em. If $f:N\to M$ is an analytic map between real analytic manifolds $N$ and $M$ and $S:=\{g_1>0,\ldots,g_r>0,h=0\}\subset M$ is a basic $C$-semianalytic set where $g_i,h\in\an(M)$,
$$
f^{-1}(S)=\{x\in M:\ f(x)\in N\}=\{g_1\circ f>0,\ldots,g_r\circ f>0,h\circ f=0\}\subset N
$$
is a basic $C$-semianalytic set where $g_i\circ f,h\circ f\in\an(N)$. 

\paragraph{}\label{cl}\em Closure\em. To prove that the closure of a $C$-semianalytic set is a $C$-semianalytic set, we observe first that if $M=\bigcup_{j\in J}U_j$ is an open covering of $M$, then 
$$
\cl(S)=\bigcup_{j\in J}\cl(S)\cap U_j=\bigcup_{j\in J}\cl(S\cap U_j)\cap U_j.
$$
Thus, it is enough to check that for each $x\in M$ there exists a neighborhood $U^x$ of $x$ in $M$ such that $\cl(S\cap U^x)\cap U^x$ is a global $C$-semianalytic set. Take an open $C$-semianalytic neighborhood of $x\in M$ such that $\cl(U^x)$ is a compact $C$-semianalytic set and apply \cite[VIII.7.2]{abr}. 

\paragraph{}\em Interior\em. It is enough to use that the complement and the closure of a $C$-semianalytic set are $C$-semianalytic sets.

\paragraph{}\label{cc}\em Connected components\em. For each $x\in M$ there exists by \cite[VIII.7.5]{abr} an open basic $C$-semianalytic set $U^x$ such that $S\cap U^x$ is the union of $m$ connected $C$-semianalytic sets and $m$ coincides with the number of connected components of the germ $S_x$. Let now $C$ be a connected component of $S$. For each $x\in M$ pick a neighborhood $U^x$ as above and observe that $C\cap U$ is a finite union of the family of the connected components of $S\cap U$, so it is a global $C$-semianalytic set. Consequently, $C$ is a $C$-semianalytic set, as required.
\end{proof}

A classic result for semialgebraic sets states: \em each semialgebraic set $S\subset\R^n$ that is in addition open admits a description of the type $S=\bigcup_{i=1}^r\{g_{i1}>0,\ldots,g_{ir}>0\}$ where $g_{ij}\in\R[\x]:=\R[\x_1,\ldots,\x_n]$. If $S$ is in addition closed, it can be written as $S=\bigcup_{i=1}^r\{h_{i1}\geq0,\ldots,h_{ir}>0\}$ where $h_{ij}\in\R[\x]$\em. By \cite[3.1]{abs} the previous results extend to open and closed global $C$-semianalytic sets. 

\begin{lem}[Homogeneity of open and closed $C$-semianalytic sets]\label{hop}
Let $S\subset M$ be a $C$-semi\-analytic set. We have:
\begin{itemize}
\item[(i)] If $S$ is open, for each $x\in M$ there exists an open neighborhood $U^x\subset M$ of $x$ and analytic functions $g_{ij}\in\an(M)$ such that $S\cap U^x=\bigcup_{i=1}^r\{g_{i1}>0,\ldots,g_{ir}>0\}$.
\item[(ii)] If $S$ is closed, for each $x\in M$ there exists an open neighborhood $U^x\subset M$ of $x$ and analytic functions $g_{ij}\in\an(M)$ such that $S\cap\cl(U^x)=\bigcup_{i=1}^r\{g_{i1}\geq0,\ldots,g_{ir}\geq0\}$.
\end{itemize}
\end{lem}
\begin{proof}
Fix a point $x\in M$ and let $U^x\subset M$ be an open neighborhood of $x$ such that $S\cap U^x$ and $S\cap\cl(U^x)$ are global $C$-semianalytic set. The statement follows from \cite[3.1]{abs} applied to $S\cap U^x$ and $S\cap\cl(U^x)$. 
\end{proof}

\subsection{Stronger results concerning closure and connected components} 
For later purposes we will need stronger results concerning the closure and the connected components of $C$-semianalytic sets. Let $(X,\an_X)$ be a reduced Stein space endowed with an anti-involution $\sigma$ such that its fixed part space $X^\sigma$ is non-empty and let $\u:=(\u_1,\ldots,\u_m)$, where $m\geq0$, be a tuple of variables. Denote either $\an(X^\sigma)[\u]$ or ${\mathcal A}(X^\sigma)[\u]$ with ${\mathscr A}$. A $C$-semianalytic set $S\subset X^\sigma\times\R^m$ is \em ${\mathscr A}$-definable \em if for each $x\in X^\sigma$ there exists an open neighborhood $U^x$ such that $S\cap U^x$ is a finite union of sets of the type $\{f=0,g_1>0,\ldots,g_r>0\}$ where $f,g_i\in{\mathscr A}$. We define analogously ${\mathscr A}$-definable global and basic $C$-semianalytic sets.

\begin{prop}\label{salvation}
Let $S\subset X^\sigma\times\R^m$ be either a bounded global $C$-semianalytic set or a $C$-semianalytic set. Assume $S$ is ${\mathscr A}$-definable. Then 
\begin{itemize}
\item[(i)] For each $x\in X^\sigma$ there exists an open ${\mathscr A}$-definable basic $C$-semianalytic set $U^x$ such that $S\cap U^x$ is the union of $m$ connected ${\mathscr A}$-definable $C$-semianalytic sets and $m$ coincides with the number of connected components of the germ $S_x$.
\item[(ii)] The closure of $S$ and its connected components are 
\begin{itemize}
\item[(ii.1)] ${\mathscr A}$-definable global $C$-semianalytic sets in the first case.
\item[(ii.2)] ${\mathscr A}$-definable $C$-semianalytic sets in the second case.
\end{itemize}
\end{itemize}
\end{prop}
\begin{proof}
We prove first (i). As $S$ is bounded, its closure $\cl(S)$ is compact, so we are reduced to prove the result for a small enough open neighborhood of each point of $\cl(S)$. Notice that each point of $X^\sigma$ has a basis of open ${\mathcal A}(X^\sigma)$-definable basic $C$-semianalytic neighborhoods in $X^\sigma$. This is clear if $X$ is a complex analytic subset of $\C^n$ and it follows from Lemmas \ref{excel-conv} and \ref{transfer} if $(X,\an_X)$ is a general reduced Stein space.

Fix a point $(x_0,u_0)\in\cl(S)$ and let $W$ be an small enough open ${\mathscr A}$-definable basic $C$-semianalytic neighborhood of $(x_0,u_0)$ in $X^\sigma\times\R^m$. Consider the bounded global $C$-semianalytic set $S\cap W$. By Lemmas \ref{excel-conv} and \ref{transfer}, we may assume $X$ is a complex analytic subset of $\C^n$, so $S\cap W$ is an ${\mathscr A}'$-definable global $C$-semianalytic subset of $\R^n\times\R^m$, where ${\mathscr A}'$ is either $\an(\R^n)[\u]$ or ${\mathcal A}(\R^n)[\u]$.

By Lemmas \ref{excell} and \ref{excellpol} one can reproduce `verbatim' the proofs of \cite[VIII.7.1-2 \& VIII.7.5-6]{abr} to prove (i) and (ii.1), substituting everywhere the local excellent ring $\an(\R^n)_{\gtm_{x_0}}$ by the local excellent ring ${\mathscr A}'_{\gtg_{(x_0,u_0)}}$ where $\gtg_{(x_0,u_0)}$ is the maximal ideal of ${\mathscr A}'$ associated to $(x_0,u_0)$.

To prove (ii.2) one proceeds as in the proof of \ref{cl} and \ref{cc} using (i) and (ii.1) instead of \cite[VIII.7.2 \& VIII.7.5]{abr}.
\end{proof}

\subsection{Dimension of a $C$-semianalytic set}
Let $S\subset M$ be a $C$-semianalytic set. We define the dimension of $S$ as $\dim(S):=\sup_{x\in M}\{\dim(S_x)\}$. For the dimension of semianalytic germs see \cite[VIII.2.11]{abr}. Given a subset $E\subset M$, we denote the Zariski closure of $E$ in $M$ (that is, the smallest $C$-analytic subset of $M$ that contains $E$) with $\ol{E}^{\zar}$. It is an easy exercise to check the following. 

\begin{lem}\label{dimgcss}
Let $S\subset M$ be a global $C$-semianalytic set. Then $\dim(\ol{S}^{\zar})=\dim(S)$.
\end{lem}

We prove next that the set of points of dimension $k$ of a $C$-semianalytic set $S$ is a $C$-semianalytic set.

\begin{prop}\label{dimks}
Let $S\subset M$ be a $C$-semianalytic set and let $S_{(k)}$ be the subset of $S$ of points of (local) dimension $k$. Then $S_{(k)}$ is a $C$-semianalytic set for all $k\geq0$.
\end{prop}
\begin{proof}
Let $d\leq\dim(M)$ be the dimension of $S$ and let us check that $S_{(d)}$ is $C$-semianalytic. 

Fix $x\in\cl(S_{(d)})$ and let $U^x\subset M$ be a neighborhood of $x$ such that $S\cap U^x=\bigcup_{i=1}^rS_i$ where each $S_i$ is a basic $C$-semianalytic set. Assume that $\dim(S_{i,x})=d$ exactly for $i=1,\ldots,\ell$. Let $T_i$ be the union of the connected components of $S_i\setminus\Sing(\ol{S_i}^{\zar})$ of dimension $d$. By \ref{cc} $T_i$ is a $C$-semianalytic set. The set of points of dimension $d$ of $S_i$ is $\cl(T_i)\cap S_i$. Thus,
$$
S_{(d)}\cap U^x=\bigcup_{i=1}^\ell\cl(T_i)\cap S_i
$$
is a $C$-semianalytic set, so $S_{(d)}$ is $C$-semianalytic.

Next, consider the $C$-semianalytic set $S':=S\setminus S_{(d)}$ and let $d'<d$ be the dimension of $S'$. Notice that $S_{(d')}=S'_{(d')}$, so $S_{(d')}$ is a $C$-semianalytic set. Proceeding recursively we conclude that each $S_{(k)}$ is $C$-semianalytic for $k\geq0$, as required.
\end{proof}

We end this section with a characterization of $C$-semianalytic sets of dimension $\leq k$ inspired by \cite[2.14]{bm} and \cite[17. Prop.7, pag. 60]{l1}.
 
\begin{prop}\label{lzc0}
Let $S\subset M$. Then $S$ is $C$-semianalytic of dimension $\leq k$ if and only if for each $x\in M$ there exist an open $C$-semianalytic neighborhood $U^x$ of $x$ in $M$ and a $C$-analytic set $Z$ of dimension $\leq k$ that contains $S\cap U^x$ such that $\cl(S\cap U^x)\setminus(S\cap U^x)$ and $(S\cap U^x)\setminus\Int_Z(S\cap U^x)$ are $C$-semianalytic sets of dimension $\leq k-1$.
\end{prop}

The proof of Proposition \ref{lzc0} requires a preliminary result inspired by \cite[2.15]{bm}.

\begin{lem}\label{diff0}
Let $S\subset T\subset M$ where $T$ is $C$-semianalytic. Let $S_1:=\cl(S)\cap T$ (be the closure of $S$ in $T$) and let $S_2:=S\setminus\cl(T\setminus S)$ (be the interior of $S$ in $T$). Then $S$ is $C$-semianalytic if and only if $S_1\setminus S$ and $S\setminus S_2$ are $C$-semianalytic.
\end{lem}
\begin{proof}
The `only if' implication is clear. To prove the `if', we proceed as follows: $T\setminus S_1$ and $S_2$ are disjoint open and closed subsets of their union $T\setminus(S_1\setminus S_2)$, which is $C$-semianalytic because $S_1\setminus S_2=(S_1\setminus S)\cup(S\setminus S_2)$ is a $C$-semianalytic set. Therefore, $S_2$ is $C$-semianalytic because it is the union of some of the connected components of $T\setminus(S_1\setminus S_2)$, which are $C$-semianalytic sets. Then $S=S_2\cup(S\setminus S_2)$ is $C$-semianalytic. 
\end{proof}

Recall that if $S$ is a semianalytic set, $\cl(S)_x=\cl(S_x)$ and $\dim(\cl(S_x)\setminus S_x)<\dim(S_x)$ for each $x\in\cl(S)$ (see \cite[VIII.2.11]{abr}). We are ready to prove Proposition \ref{lzc0}.

\begin{proof}[Proof of Proposition \em\ref{lzc0}]
We prove the `if' implication first. We have to check that $S\cap U^x$ is $C$-semianalytic of dimension $\leq k$ for all $x\in M$. By Lemma \ref{diff0} $S\cap U^x$ is $C$-semianalytic. On the other hand, $S\cap U^x=\Int_Z(S\cap U^x)\cup(S\cap U^x\setminus\Int_Z(S\cap U^x))$ and it has dimension $\leq k$ because $\Int_Z(S\cap U^x)\subset Z$ has dimension $\leq k$ and $S\cap U^x\setminus\Int_Z(S\cap U^x)$ has dimension $\leq k-1$.

We prove next the `only if' implication. We show first that \em each point $x\in M$ has an open $C$-semianalytic neighborhood $U^x\subset M$ such that $\dim(S_x)=\dim(S\cap U^x)=\dim(\ol{S\cap U^x}^{\zar})$ and $\cl(S)\cap U^x$ is a global $C$-semianalytic set\em. 

Let $U^x$ be an open neighborhood of $x$ in $M$ such that $S\cap U^x$ is a finite union of basic $C$-semianalytic sets $S_1,\ldots,S_r$. After shrinking $U^x$ we may assume $x\in\cl(S_i)$, $\cl(S)\cap U^x$ is a global $C$-semianalytic set and $\dim(S_i\cap U^x)=\dim(S_{i,x})$ for $i=1,\ldots,r$. We may assume that $U^x$ is in addition basic $C$-semianalytic, so by Lemma \ref{dimgcss} $\dim(\ol{S_i\cap U^x}^{\zar})=\dim(S_i\cap U^x)$. Consequently,
$$
\dim(S_x)=\max_{i=1,\ldots,r}\{\dim(S_{i,x})\}=\max_{i=1,\ldots,r}\{\dim(S_i\cap U^x)\}=\max_{i=1,\ldots,r}\{\ol{S_i\cap U^x}^{\zar}\}=\dim(\ol{S\cap U^x}^{\zar}),
$$
so $\dim(S_x)=\dim(S\cap U^x)=\dim(\ol{S\cap U^x}^{\zar})$.

Denote $Z:=\ol{S\cap U^x}^{\zar}$ and $A:=S\cap U^x$. We have $\dim(Z)=\dim(A)\leq k$ and 
$$
\dim(\cl(A)_y\setminus A_y)\leq k-1
$$ 
for each $y\in\cl(A)$. We claim: \em $\dim(A_y\setminus\Int_Z(A)_y)\leq k-1$ for each $y\in\cl(A)$\em. 

As $\Int_Z(A)=Z\setminus\cl(Z\setminus A)$,
$$
A\setminus\Int_Z(A)=A\cap(Z\setminus\Int_Z(A))=A\cap\cl(Z\setminus A)=\cl(Z\setminus A)\setminus(Z\setminus A).
$$
For each $y\in\cl(A)$ we have
$$
\dim(A_y\setminus\Int_Z(A)_y)=\dim(\cl((Z\setminus A)_y)\setminus(Z\setminus A)_y)<\dim((Z\setminus A)_y)\leq\dim(Z)\leq k.
$$
Thus, $\cl(A)\setminus A$ and $A\setminus\Int_Z(A)$ are $C$-semianalytic sets of dimension $\leq k-1$, as required.
\end{proof}

\section{Images of $C$-semianalytic sets under proper holomorphic maps}\label{s4}

The main purpose of this section is to prove Theorems \ref{properint} and \ref{finite76}. Before that we present two enlightening examples. The first one prove that a proper real analytic map with finite fibers may not admit a proper complex holomorphic extension to complexifications. The second one shows that Theorem \ref{properint}(iii) can be false if $X^\sigma\subsetneq F^{-1}(Y^\tau)$.

\begin{examples}
(i) Consider the map $f:\R^2\to\R^2,\ (x,y)\mapsto((x^2+y^2)x,(x^2+y^2)y)$. Observe that $f$ is an analytic map, with continuous inverse map: 
$$
f^{-1}:\R^2\to\R^2, (u,v)\mapsto
\begin{cases}
(\frac{u}{\sqrt[3]{u^2+v^2}},\frac{v}{\sqrt[3]{u^2+v^2}})&\text{if $(u,v)\neq(0,0)$,}\\
(0,0)&\text{if $(u,v)=(0,0)$.}
\end{cases}
$$
Thus, $f$ is a homeomorphism, so it is a proper analytic map with finite fibers. However, $f$ cannot be extended to a proper holomorphic map $F:U\to\C^2$ defined on an open neighborhood $U$ of $\R^2$ in $\C^2$. 

(ii) Let $Y:=\{x^2-zy^2=0\}\subset\C^3$ be the complex Whitney's umbrella endowed with the complex conjugation $\tau$, so $Y^\tau=\{x^2-zy^2=0\}\subset\R^3$. Let $X=\C^2$ endowed with the complex conjugation $\sigma$. Consider the finite holomorphic map
$$
F:X:=\C^2\to Y,\ (s,t)\mapsto(st,t,s^2).
$$
and let $Z=X^\sigma=\R^2$. Observe that $X^\sigma\subsetneq F^{-1}(Y^\tau)$ and $F(Z)=\{x^2-zy^2=0\}\setminus\{z<0\}\subset\R^3$, which is not a $C$-analytic subset of $Y^\tau$.
\end{examples}

\subsection{Local approach}\setcounter{paragraph}{0}
Let us prove Theorem \ref{finite76}.

\begin{proof}[Proof of Theorem \em \ref{finite76}]
Consider the sheaf $F_*(\an_X)$ of $\an_Y$-modules. Recall that if $V\subset Y$ is open, $H^0(V,F_*(\an_X))=H^0(F^{-1}(V),\an_X)$. Fix $y_0\in Y$ and write $F^{-1}(y_0):=\{x_1,\ldots,x_\ell\}$. By \cite[6.1.18]{jp}
$$
F_*(\an_X)_{y_0}=\prod_{i=1}^\ell\an_{X,x_i}=\prod_{i=1}^\ell\an(X_{x_i})
$$
and by \cite[6.3.5]{jp} $F_*(\an_X)$ is a coherent sheaf of $\an_Y$-modules. By Cartan's Theorem A there exist $H_1,\ldots,H_m\in H^0(Y,F_*(\an_X))=H^0(X,\an_X)=\an(X)$ such that $F_*(\an_X)_{y_0}$ is generated by $H_1,\ldots,H_m$ as a $\an_{Y,y_0}$-module. By \cite[VIII.4.4]{abr} we have the following diagram of faithfully flat homomorphisms between local excellent rings
$$
\xymatrix{
\an(X)_{\gtm_{x_i}}\ar@{^{(}->}[r]\ar@{^{(}->}[d]&\an(X_{x_i})\ar@{^{(}->}[d]\\
\widehat{\an(X)_{\gtm_{x_i}}}\ar[r]^{\cong}&\widehat{\an(X_{x_i})}
}
$$
Consider the inclusion of $F^*(\an(Y)_{\gtn_{y_0}})$-modules
$$
M_1:=\sum_{i=1}^mH_i\cdot F^*(\an(Y)_{\gtn_{y_0}})\subset M_2:=F^*(\an(Y)_{\gtn_{y_0}})[H_1,\ldots,H_m]\subset\Ss^{-1}(\an(X)).
$$
Recall that $\Ss:=\an(X)\setminus(\gtm_{x_1}\cup\cdots\cup\gtm_{x_\ell})$.

\paragraph{}We have to prove: $M_1=M_2=B:=\Ss^{-1}(\an(X))$.

In what follows, completions of rings are considered with respect to the Jacobson ideal. In case the involved ring is local, its Jacobson ideal coincides with its unique maximal ideal. The maximal ideals of the semi-local ring $B$ are $\gtm_{x_1}B,\ldots,\gtm_{x_\ell}B$. By \cite[24.C, p. 174]{m} the completion of $B$ satisfies $\widehat{B}\cong\prod_{i=1}^\ell\widehat{B_{\gtm_{x_i}B}}$. Observe that $B_{\gtm_{x_i}B}\cong\an(X)_{\gtm_{x_i}}$, so 
\begin{align*}
&B\hookrightarrow\prod_{i=1}^\ell B_{\gtm_{x_i}B}\cong\prod_{i=1}^\ell\an(X)_{\gtm_{x_i}}\hookrightarrow\prod_{i=1}^\ell\an(X_{x_i})=F_*(\an_X)_{y_0},\\
&\widehat{B_{\gtm_{x_i}B}}\cong\widehat{\an(X)_{\gtm_{x_i}}}\cong\widehat{\an(X_{x_i})}. 
\end{align*}
By \cite[24.C, p. 174]{m} the completion of the semi-local ring $F_*(\an_X)_{y_0}$ is
$$
\widehat{F_*(\an_X)_{y_0}}\cong\prod_{i=1}^\ell\widehat{\an(X_{x_i})}\cong\prod_{i=1}^r\widehat{B_{\gtm_{x_i}B}}\cong\widehat{B}.
$$

As $F_*(\an_X)_{y_0}$ is a finitely generated $\an_{Y,y_0}$-module, the completion of $F_*(\an_X)_{y_0}$ with respect to the maximal ideal of $\an_{Y,y_0}$ is by \cite[23.K-L, Thm.55, p.170]{m} 
$$
\widehat{F_*(\an_X)_{y_0}}\cong F_*(\an_X)_{y_0}\otimes_{\an_{Y,y_0}}\widehat{\an_{Y,y_0}}. 
$$
Recall that $F_*(\an_X)_{y_0}$ is generated by $H_1,\ldots,H_m$ as a $\an_{Y,y_0}$-module, so $\widehat{F_*(\an_X)_{y_0}}$ is generated by $H_1,\ldots,H_m$ as a $\widehat{\an_{Y,y_0}}$-module. Again by \cite[23.K-L, Thm.55, p.170]{m} the completion of $M_1$ with respect to the maximal ideal of $\an(Y)_{\gtn_{y_0}}$ is 
$$
\widehat{M_1}\cong M_1\otimes_{\an(Y)_{\gtn_{y_0}}}\widehat{\an(Y)_{\gtn_{y_0}}}.
$$
As $M_1$ is generated by $H_1,\ldots,H_m$ as a $\widehat{\an(Y)_{\gtn_{y_0}}}$-module, $\widehat{M_1}$ is generated by $H_1,\ldots,H_m$ as a $\widehat{\an(Y)_{\gtn_{y_0}}}$-module. Consequently, since $\widehat{\an(Y)_{\gtn_{y_0}}}\cong\widehat{\an_{Y,y_0}}$, we conclude $\widehat{M_1}\cong\widehat{F_*(\an_X)_{y_0}}$. As the inclusion $\an(Y)_{\gtn_{y_0}}\hookrightarrow\widehat{\an(Y)_{\gtn_{y_0}}}$ is faithfully flat because $\an(Y)_{\gtn_{y_0}}$ is a local excellent ring, we have the following commutative diagram of $\widehat{\an(Y)_{\gtn_{y_0}}}$-modules
$$
\xymatrix{
M_1\otimes_{\an(Y)_{\gtn_{y_0}}}\widehat{\an(Y)_{\gtn_{y_0}}}\ar@{^{(}->}[r]\ar@{<->}^\cong[d]& M_2\otimes_{\an(Y)_{\gtn_{y_0}}}\widehat{\an(Y)_{\gtn_{y_0}}}\ar@{^{(}->}[r]&B\otimes_{\an(Y)_{\gtn_{y_0}}}\widehat{\an(Y)_{\gtn_{y_0}}}\ar@{^{(}->}[d]\\
\widehat{M_1}\ar@{<->}^(0.425)\cong[r]&\widehat{F_*(\an_X)_{y_0}}\ar@{<->}^(0.35)\cong[r]&F_*(\an_X)_{y_0}\otimes_{\an_{Y,y_0}}\widehat{\an_{Y,y_0}}
}
$$
Thus, all the inclusion are in fact isomorphisms, so
$$
M_1\otimes_{\an(Y)_{\gtn_{y_0}}}\widehat{\an(Y)_{\gtn_{y_0}}}=M_2\otimes_{\an(Y)_{\gtn_{y_0}}}\widehat{\an(Y)_{\gtn_{y_0}}}=B\otimes_{\an_{Y,y_0}}\widehat{\an_{Y,y_0}}.
$$
As the inclusion $\an(Y)_{\gtn_{y_0}}\hookrightarrow\widehat{\an(Y)_{\gtn_{y_0}}}$ is faithfully flat, $M_1=M_2=B$, as required.
\end{proof}
\begin{remark}\label{finite76r}
If $F:(X,\an_X)\to(Y,\an_Y)$ is in addition invariant, we assume that each $H_i$ is invariant after taking $\Re(H_i)$ and $\Im(H_i)$ instead of $H_i$.
\end{remark}

\subsection{Global approach}\label{bmp}\setcounter{paragraph}{0}
A key result to prove Theorem \ref{properint} is \cite[Thm.2.2]{bm} that we translate next to our situation. Let $(X,\an_X)$ be a reduced Stein space endowed with an anti-involution $\sigma$ such that its fixed part space $X^\sigma$ is non-empty and let $\u:=(\u_1,\ldots,\u_m)$, where $m\geq0$, be a tuple of variables. Denote either $\an(X^\sigma)$ or ${\mathcal A}(X^\sigma)$ with ${\mathscr A}$.

\begin{thm}\label{bm}
Let $S\subset X^\sigma\times\R^m$ be an ${\mathscr A}[\u]$-definable global $C$-semianalytic set and consider the projection $\pi:E\times\R^k\to\R^k,\ (x,u)\mapsto x$ onto the first factor. Then $\pi(S)$ is an ${\mathscr A}$-definable global $C$-semianalytic set.
\end{thm}

We also need the following application of M. Artin's approximation theorem, that we include in detail for the sake of completeness.
\begin{lem}\label{artin}
Let $X_0\subset\C^n_0$ and $Y_0\subset\C^m_0$ be complex analytic set germs at the origin. Let $\varphi:\an(X_0)\to\an(Y_0)$ be a local analytic homomorphism such that the extension to the completions
$\widehat{\varphi}:\widehat{\an(X_0)}\to\widehat{\an(Y_0)}$ is an isomorphism. Then $\varphi$ is also an isomorphism.
\end{lem}
\begin{proof}
Denote the variables of $\C^n$ with $(x_1,\ldots,x_n)$ and the variables of $\C^m$ with $(y_1,\ldots,y_m)$. Let $G_i:=\varphi(x_i)$ for $i=1,\ldots,n$. By the universal property of local homomorphisms between power series rings, $\widehat{\varphi}$ is given by
$$
\widehat{\varphi}:\widehat{\an(X_0)}\to\widehat{\an(Y_0)},\ \zeta\mapsto\zeta(G_1,\ldots,G_n).
$$
In particular, $\varphi:\an(X_0)\to\an(Y_0),\ F\mapsto F(G_1,\ldots,G_n)$. Denote $\eta_j:=(\widehat{\varphi})^{-1}(y_j)\in\widehat{\an(X_0)}$ for $j=1,\ldots,m$. As before, $(\widehat{\varphi})^{-1}$ is given by
$$
(\widehat{\varphi})^{-1}:\widehat{\an(Y_0)}\to\widehat{\an(X_0)},\ \xi\mapsto\xi(\eta_1,\ldots,\eta_n).
$$
In particular, the tuple $(\eta_1,\ldots,\eta_m)\in(\widehat{\an(X_0)})^m$ is a solution of the system
\begin{equation}\label{system}
\begin{cases}
G_1(\y_1,\ldots,\y_m)=x_1,\\
\hspace{1.75cm}\vdots\\
G_n(\y_1,\ldots,\y_m)=x_m.
\end{cases}
\end{equation}
By M. Artin's approximation theorem, there exist a solution $(F_1,\ldots,F_m)\in\an(X_0)^m$ of \eqref{system}. Consider the local homomorphism
$$
\psi:\widehat{\an(Y_0)}\to\widehat{\an(X_0)},\ \xi\mapsto\xi(F_1,\ldots,F_m),
$$
that satisfies $\psi\circ\widehat{\varphi}(x_i)=\psi(G_i)=x_i$. Thus, $\psi\circ\widehat{\varphi}=\id_{\widehat{\an(X_0)}}$, so $\psi=(\widehat{\varphi})^{-1}$. It follows that 
$$
\psi|_{\an(Y_0)}:\an(Y_0)\to\an(X_0),\ H\mapsto H(F_1,\ldots,F_m)
$$ 
is the inverse of $\varphi$, so $\varphi$ is an isomorphism, as required.
\end{proof}

Next, we prove Theorem \ref{properint}. 

\begin{proof}[Proof of Theorem \em \ref{properint}]
Recall that $F:(X,\an_X)\to(Y,\an_Y)$ is an invariant proper holomorphic map and $S\subset X^\sigma$ is an ${\mathcal A}(X^\sigma)$-definable $C$-semianalytic set. As $F$ is proper, we may assume for the whole proof that $X$ is irreducible and $F$ is surjective because if $\{X_\alpha\}_\alpha$ is the locally finite family of the irreducible components of $X$, it holds that $\{Y_\alpha:=F(X_\alpha)\}_\alpha$ is by Remmert's Theorem a locally finite family of irreducible complex analytic subsets of $Y$.

We prove together (i) and (ii). We have already denoted $E:=\cl(F^{-1}(Y^\tau)\setminus X^\sigma)$. Fix a point $y_0\in Y^\tau$. We have to show: \em There exists an open neighborhood $A\subset Y^\tau$ of $y_0$ such that $F(S)\cap A$ and $F(E\cap S)\cap A$ are global $C$-semianalytic sets\em.

The proof is conducted in several steps:
 
\paragraph{}\label{prep}
Write $F^{-1}(y_0):=\{x_1,\ldots,x_\ell\}$ and let $\Ss:=\an(X)\setminus(\gtm_{x_1}\cup\cdots\cup\gtm_{x_\ell})$. By Theorem \ref{finite76} and Remark \ref{finite76r} there exist invariant $H_1,\ldots,H_m\in\an(X)$ such that 
$$
\Ss^{-1}(\an(X))=F^*(\an(Y)_{\gtn_{y_0}})[H_1,\ldots,H_m].
$$
Consider the evaluation epimorphism
$$
\theta:\an(Y)_{\gtn_{y_0}}[\z_1,\ldots,\z_m]\to\Ss^{-1}(\an(X)),\ Q(y,\z_1,\ldots,\z_m)\mapsto Q(F(x),H_1(x),\ldots,H_m(x))
$$
that maps $y$ to $F(x)$ and the variables $\z_j$ to $H_j(x)$.

As $X$ is irreducible, $\Ss^{-1}(\an(X))$ is an integral domain. By the first isomorphism theorem there exist a prime ideal $\gtp$ of $\an(Y)_{\gtn_{y_0}}[\z_1,\ldots,\z_m]$ and an isomorphism 
$$
\ol{\theta}:\an(Y)_{\gtn_{y_0}}[\z_1,\ldots,\z_m]/\gtp\to\Ss^{-1}(\an(X)),\ Q(\z_1,\ldots,\z_m)+\gtp\mapsto Q(F(x),H_1(x),\ldots,H_m(x)).
$$
As the ring $\an(Y)_{\gtn_{y_0}}[\z_1,\ldots,\z_m]$ is noetherian because $\an(Y)_{\gtn_{y_0}}$ is excellent, $\gtp$ is finitely generated. Let $P_1,\ldots,P_t\in\an(Y)[\z_1,\ldots,\z_m]$ be a system of generators of $\gtp$. Let 
$$
\Omega:=\{(y,z)\in Y\times\C^m: P_1(y,z)=0,\ldots,P_t(y,z)=0\}\subset Y\times\C^m
$$ 
and consider 
$$
H:X\to\Omega\subset Y\times\C^m,\ x\mapsto(F(x),H_1(x),\ldots,H_m(x)),
$$
which satisfies $\ol{\theta}=H^*$. We use implicitly the irreducibility of $X$ here to assure that $H(X)\subset\Omega$. Denote $\omega_i:=H(x_i)$ and observe that the maximal ideals of the semi-local ring $A:=\an(Y)_{\gtn_{y_0}}[\z_1,\ldots,\z_m]/\gtp$ are the maximal ideals $\gtn_{\omega_i}$ of $A$ associated to the point $\omega_i$ for $i=1,\ldots,\ell$. Denote $B:=\Ss^{-1}(\an(X))$ and observe that $H^*$ induces a isomorphism between the local rings $A_{\gtn_{\omega_i}}$ and $B_{\gtm_{x_i}B}\cong\an(X)_{\gtm_{x_i}}$, where $\gtm_{x_i}$ is the maximal ideal of $\an(X)$ associated to the point $x_i$ for $i=1,\ldots,\ell$. Thus, $H^*$ induces an isomorphism $\widehat{H^*}:\widehat{A_{\gtn_{\omega_i}}}\to\widehat{B_{\gtm_{x_i}B}}$ between the corresponding completions. As $(X,\an_X)$ and $(Y,\an_Y)$ are reduced Stein spaces, $\widehat{\an(X)_{\gtm_{x_i}}}$ coincides with $\widehat{\an(X_{x_i})}$ and $\widehat{A_{\gtn_{\omega_i}}}$ coincides with $\widehat{\an(\Omega_{\omega_i})}$. Thus, we have the following isomorphism
$$
\xymatrix{
\widehat{\an(\Omega_{\omega_i})}\cong\widehat{A_{\gtn_{\omega_i}}}\ar[rr]^(0.4){\widehat{H^*}=\widehat{\ol{\theta}}}_(0.4)\cong&&\widehat{B_{\gtm_{x_i}B}}\cong\widehat{\an(X)_{\gtm_{x_i}}}\cong\widehat{\an(X_{x_i})}
}
$$
for $i=1,\ldots,\ell$. By Lemma \ref{artin}, the homomorphism $\an(\Omega_{\omega_i})\to\an(X_{x_i}),\ G_{\omega_i}\mapsto G_{\omega_i}\circ H_{x_i}$ is an isomorphism for $i=1,\ldots,\ell$. Consequently, there exist invariant open neighborhoods $U$ of $\{x_1,\ldots,x_\ell\}$ in $X$ and $V$ of $\{\omega_1,\ldots,\omega_\ell\}$ in $\Omega$ such that $H^{-1}(V)=U$ and $H|_U:U\to V$ is an invariant complex analytic diffeomorphism.

\paragraph{} As each $H_i$ is invariant, also $H$ is invariant, so $H(X^\sigma)\subset\Omega^{\tau'}:=\Omega\cap(Y^\tau\times\R^m)$, where 
$$
\tau':Y\times\C^m\to Y\times\C^m,\ (y,z)\mapsto(\tau(y),\ol{z})
$$ 
is an anti-involution. Consider the restriction maps $f:=F|_{X^\sigma}:X^\sigma\to Y^\tau$ and $h:=H|_{X^\sigma}:X^\sigma\to\Omega^{\tau'}$. As $S$ is ${\mathcal A}(X^\sigma)$-definable, there exist an open neighborhood $W\subset U\cap X^\sigma$ of the finite set $f^{-1}(y_0)=F^{-1}(y_0)\cap X^\sigma$ such that $S\cap W$ is an ${\mathcal A}(X^\sigma)$-definable global $C$-semianalytic set, that is, $S\cap W=\bigcup_{i=1}^r\bigcap_{j=1}^sS_{ij}$ where $S_{ij}$ is either 
\begin{equation}\label{accor}
\{x\in X^\sigma:\ f_{ij}(x)>0\}\quad\text{or}\quad\{x\in X^\sigma:\ f_{ij}(x)=0\}. 
\end{equation}
and $f_{ij}:=F_{ij}|_{X^\sigma}$. In addition $h|_W:W\to h(W)$ is a real analytic diffeomorphism, $h^{-1}(h(W))=W$ and $h(W)$ is an open neighborhood of $h(f^{-1}(y_0))$ in $\Omega^{\tau'}$. As $\ol{\theta}=H^*$ is an isomorphism, we may assume by Lemma \ref{transfer} 
\begin{itemize}
\item $f_{ij}=q_{ij}\circ h$ for some $q_{ij}\in\an(Y^\tau)[\u_1,\ldots,\u_m]$,
\item $h(S\cap W)=T$ where $T:=\bigcup_{i=1}^r\bigcap_{j=1}^sT_{ij}\subset\Omega^{\tau'}$ and $T_{ij}$ is either 
$$
\{(y,u)\in\Omega^{\tau'}:\ q_{ij}(y,u)>0\}\quad\text{or}\quad\{(y,u)\in\Omega^{\tau'}:\ q_{ij}(y,u)=0\} 
$$
accordingly to the choices of signs done in \eqref{accor}. 
\end{itemize}
Let $\pi:Y^\tau\times\R^m\to Y^\tau, (y,u)\mapsto y$ be the projection onto the first factor. As $f=\pi\circ h$, we conclude by Theorem \ref{bm} $f(S\cap W)=\pi(h(S\cap W))=\pi(T)$ is a global $C$-semianalytic set.

\paragraph{}Observe that 
\begin{multline*}
H(F^{-1}(Y^\tau)\setminus X^\sigma)=\{(y,u+\sqrt{-1}v)\in Y^\tau\times\C^m:\\ 
P_1(y,u+\sqrt{-1}v)=0,\ldots,P_t(y,u+\sqrt{-1}v)=0, v_1^2+\cdots+v_m^2\neq0\}.
\end{multline*}
Thus, we can understand $H(F^{-1}(Y^\tau)\setminus X^\sigma)$ as an $\an(Y^\tau)[\u,{\tt v}]$-definable global $C$-semianalytic subset of $Y^\tau\times\R^m\times\R^m$ where $\u:=(\u_1,\ldots,\u_m)$ and ${\tt v}:=({\tt v}_1,\ldots,{\tt v}_m)$. Let $V'\subset V\cap(Y^\tau\times\R^m\times\R^m)$ be a bounded open $\an(Y^\tau)[\u,{\tt v}]$-definable global $C$-semianalytic neighborhood of $\{\omega_1,\ldots,\omega_\ell\}$. By Proposition \ref{salvation}
$$
\cl(H(F^{-1}(Y^\tau)\setminus X^\sigma)\cap V')
$$
is an $\an(Y^\tau)[\u,{\tt v}]$-definable global $C$-semianalytic subset of $Y^\tau\times\R^m\times\R^m$. Thus, as $h(S\cap W)\subset\{{\tt v}_1=0,\ldots,{\tt v}_m=0\}$,
$$
\cl(H(F^{-1}(Y^\tau)\setminus X^\sigma)\cap V')\cap h(S\cap W)
$$
is an $\an(Y^\tau)[\u]$-definable global $C$-semianalytic subset of $Y^\tau\times\R^m$. We may assume $W\subset H^{-1}(V')$. As $h|_W:W\to h(W)$ is a real analytic diffeomorphism,
$$
h(E\cap S\cap W)=h(\cl(F^{-1}(Y^\tau)\setminus X^\sigma)\cap S\cap W)=h(\cl(F^{-1}(Y^\tau)\setminus X^\sigma)\cap W)\cap h(S\cap W).
$$
In addition, as $H|_U:U\to V$ is an analytic diffeomorphism and $H^{-1}(V)=U$, it holds
\begin{equation*}
\begin{split}
h(\cl(F^{-1}(Y^\tau)\setminus X^\sigma)\cap W)&=H(\cl(F^{-1}(Y^\tau)\setminus X^\sigma)\cap H^{-1}(V')\cap X^\sigma\cap W)\\
&=H(\cl((F^{-1}(Y^\tau)\setminus X^\sigma)\cap H^{-1}(V'))\cap H^{-1}(V'))\cap H(X^\sigma\cap W)\\
&=\cl(H(F^{-1}(Y^\tau)\setminus X^\sigma)\cap V')\cap V'\cap H(X^\sigma\cap W).
\end{split}
\end{equation*}
Consequently,
$$
h(E\cap S\cap W)=\cl(H(F^{-1}(Y^\tau)\setminus X^\sigma)\cap V')\cap h(S\cap W)
$$
is an $\an(Y^\tau)[\u]$-definable global $C$-semianalytic subset of $Y^\tau\times\R^m$. As $f=\pi\circ h$, we conclude by Theorem \ref{bm} $f(E\cap S\cap W)=\pi(h(E\cap S\cap W))$ is a global $C$-semianalytic set.

\paragraph{}\label{finish}To finish the proof of (i) and (ii) we find \em an open $C$-semianalytic neighborhood $A\subset Y^\tau$ of $y_0$ such that $f(S)\cap A=f(S\cap W)\cap A$ and $f(E\cap S)\cap A=f(E\cap S\cap W)\cap A$\em. Consequently, \em $f(S)$ and $f(E\cap S)$ are $C$-semianalytic sets\em.

As $W$ is an open neighborhood of $f^{-1}(y_0)$ in $X^\sigma$, we have $C:=f(X^\sigma\setminus W)$ is a closed subset of $Y$ (recall that $f$ is proper) that does not contain $y_0$. Let $A\subset Y^\tau$ be an open basic $C$-semianalytic neighborhood of $y_0$ that does not intersect $C$. As
\begin{align*}
&f(S)=f(S\cap W)\cup f(S\cap(X^\sigma\setminus W))\quad\text{and}\\
&f(E\cap S)=f(E\cap S\cap W)\cup f(E\cap S\cap(X^\sigma\setminus W)),
\end{align*}
we conclude $f(S)\cap A=f(S\cap W)\cap A$ and $f(E\cap S)\cap A=f(E\cap S\cap W)\cap A$, as required.

(iii) By hypothesis there exists a complexification $\widetilde{Z}\subset X$ of $Z$ that is closed in $X$, hence $\widetilde{Z}$ is a complex analytic subset of $X$ and $\widetilde{Z}\cap X^\sigma=Z$. By Remmert's Theorem $F(\widetilde{Z})$ is a complex analytic subset of $Y$, so $F(\widetilde{Z})\cap Y^\tau$ is a $C$-analytic subset of $Y^\tau$ that contains $F(Z)$. Thus, 
$$
F(Z)\subset F(\widetilde{Z})\cap Y^\tau=F(\widetilde{Z}\cap F^{-1}(Y^\tau))=F(\widetilde{Z}\cap X^\sigma)=F(Z),
$$ 
so $F(Z)=F(\widetilde{Z})\cap Y^\tau$ is a $C$-analytic subset of $Y^\tau$, as required.
\end{proof}

\subsection{Sufficient condition}
We end this section with some sufficient conditions under which statement Theorem \ref{properint}(iii) applies.

\begin{lem}\label{situ}
Assume that $Y^\tau$ is coherent and let $F:(X,\an_X)\to(Y,\an_Y)$ be an invariant surjective proper holomorphic map. Let $Z\subset Y$ be a complex analytic set such that $X\setminus F^{-1}(Z)$ is dense in $X$. Assume that $F|_{X\setminus F^{-1}(Z)}:X\setminus F^{-1}(Z)\to Y\setminus Z$ is biholomorphic. Then $F^{-1}(Y^\tau)=X^\sigma.$
\end{lem}
\begin{proof}
We may assume $X$ is irreducible and consequently also $Y$ is irreducible. In particular, both are pure dimensional of the same dimension $d$. We have to prove $F^{-1}(Y^\tau)\subset X^\sigma$. 

Suppose by contradiction that there exists $z\in F^{-1}(Y^\tau)\setminus X^\sigma$. As $F$ is invariant, $F(\sigma(z))=F(z)\in Y^\tau$, so $\sigma(z)\in F^{-1}(Y^\tau)\setminus X^\sigma$. 

By Remmert's Theorem the germs $F(X_z)$ and $F(X_{\sigma(z)})$ are unions of irreducible components of $Y_{F(z)}$. As $X\setminus F^{-1}(Z)$ is dense in $X$ and the restriction 
$$
F|_{X\setminus F^{-1}(Z)}:X\setminus F^{-1}(Z)\to Y\setminus Z
$$ 
is biholomorphic, we conclude $\dim_\C(F(X_z)\cap F(X_{\sigma(z)}))<d$. As $F$ is invariant, $F(X_z)\cap Y^\tau=F(X_{\sigma(z)})\cap Y^\tau$, so $\dim_\R(F(X_z)\cap Y^\tau)<d$. As $Y$ is irreducible and $Y^\tau$ is coherent, $Y^\tau$ is pure dimensional, so all the irreducible components of $Y^\tau_{F(z)}$ have dimension $d$. In addition, the irreducible components of $Y^\tau_{F(z)}$ are the intersections with $Y^\tau_{F(z)}$ of the irreducible components of $Y_{F(z)}$. But this is impossible because $F(X_z)$ is a union of irreducible components of $Y_{F(z)}$, $\dim_\R(F(X_z)\cap Y^\tau)<d$ and all the irreducible components of $Y^\tau_{F(z)}$ have dimension $d$.
\end{proof}

\section{Set of points of non-coherence of a $C$-analytic set}\label{s5}

In \cite{tt} it is proved that set $N(X)$ of points where a $d$-dimensional analytic set $X$ is non-coherent is a semianalytic set. The authors sketch a smart semianalytic decomposition of the set of points $x\in X$ at which the germ $X_x$ has an irreducible component of dimension $d$ that is non-coherent. Then they prove inductively that the set of points $x\in X$ at which the germ $X_x$ has all its irreducible components of dimension $d$ coherent but one of dimension $k<d$ non-coherent is also semianalytic. In this section we improve their construction to prove that $N(X)$ is a $C$-semianalytic set whenever $X$ is $C$-analytic.

\subsection{Description of the set of points of non-coherence of a $C$-analytic set}

Let $M$ be a real analytic manifold and let $X\subset M$ be a $C$-analytic set of dimension $d$. For each $0\leq k\leq d:=\dim(X)$ let ${\mathfrak F}_k$ be the collection of all the irreducible $C$-analytic sets $Z\subset M$ of dimension $k$ that are an irreducible component of $\Sing_\ell(X)$ for some $\ell\geq0$ (see \ref{sing}). Define $Z_k:=\bigcup_{Z\in{\mathfrak F}_k}Z$ and 
$$
R_k:=\bigcup_{j=k+1}^dZ_{j,(j)}\quad\text{where}\quad Z_{j,(j)}:=\{z\in Z_j:\ \dim_{\R}(Z_j)=j\}=\cl(Z_j\setminus\Sing(Z_j)). 
$$
Let $\widetilde{Z}_k$ be a complexification of $Z_k$ and let $(Y_k,\pi_k)$ be the normalization of $\widetilde{Z}_k$ endowed with the anti-involution $\sigma_k:Y_k\to Y_k$ that is induced by the usual conjugation on $\widetilde{Z}_k$. Let 
$$
Y_k^{\sigma_k}:=\{y\in Y_k:\ \sigma_k(y)=y\} 
$$
be the fixed part space of $Y$, which is a $C$-analytic space. Define 
\begin{align*} 
&Y^{\sigma_k}_{k,(k)}:=\{y\in Y_k^{\sigma_k}:\ \dim_{\R}(Y_{k,y}^{\sigma_k})=k\}=\cl(Y^{\sigma_k}_k\setminus\Sing(Y^{\sigma_k}_k)),\\
&C_{k,1}:=\pi_k^{-1}(Z_k)\setminus Y_k^{\sigma_k},\quad C_{k,2}:=Y_k^{\sigma_k}\setminus Y^{\sigma_k}_{k,(k)},\\
&A_{k,i}:=\cl(C_{k,i})\cap\cl(Y^{\sigma_k}_{k,(k)}\setminus\pi^{-1}(R_k))\quad\text{for $i=1,2$}.
\end{align*}
We state next our main result.
 
\begin{thm}\label{ncp0}
Let $N(X)$ be the set of points of non-coherence of the $C$-analytic set $X$ of dimension $d$ and let $N_k(Z_k,R_k):=\pi_k(A_{k,1})\cup\pi_k(A_{k,2})$. Then 
\begin{itemize}
\item[(i)] $N_k(Z_k,R_k)$ is a $C$-semianalytic set of dimension $\leq k-2$.
\item[(ii)] $\bigcup_{k=j}^dN_k(Z_k,R_k)$ is the set of points of $X$ such that the germ $X_x$ has a non-coherent irreducible component of dimension $\geq j$.
\item[(iii)] $N(X)=\bigcup_{k=2}^dN_k(Z_k,R_k)$.
\end{itemize}
\end{thm}
Once we prove the previous result it holds in particular Corollary \ref{ncc}. 

\subsection{Initial preparation}\label{ip} 
By Grauert's immersion theorem we assume that $X$ is a $C$-analytic subset of $\R^n$. Let $\widetilde{X}$ be a complexification of $X$ that is an invariant complex analytic subset of a Stein open neighborhood $\Omega\subset\C^n$ of $\R^n$. Denote the restriction to $\widetilde{X}$ of the complex conjugation on $\C^n$ with $\sigma:\widetilde{X}\to\widetilde{X}$. It holds $d:=\dim_\R(X)=\dim_\C(\widetilde{X})$ and $X=\{x\in\widetilde{X}:\ \sigma(x)=x\}$. Let $\pi:Y\to\widetilde{X}$ be the normalization of $\widetilde{X}$. As $\widetilde{X}$ is Stein, also $Y$ is Stein \cite{n3}. The complex conjugation of $\widetilde{X}$ extends to an anti-involution $\widehat{\sigma}$ on $Y$ that makes the following diagram commutative \cite[IV.3.10]{gmt}
$$
\xymatrix{
&Y^{\widehat{\sigma}}\ar[d]_{\pi|_{Y^{\widehat{\sigma}}}}\ar@{^{(}->}[r]&Y\ar[r]^{\widehat{\sigma}}\ar[d]_{\pi}&Y\ar[d]^{\pi}\\
X\ar@{=}[r]&\widetilde{X}^{\sigma}\ar@{^{(}->}[r]&\widetilde{X}\ar[r]^{\sigma}&\widetilde{X}
}
$$
where $Y^{\widehat{\sigma}}:=\{y\in Y:\ \widehat{\sigma}(y)=y\}$ is the set of fixed points of $\widehat{\sigma}$.

The following results from \cite{gmt} will help us to describe the set of points of non-coherence of a $C$-analytic set. We keep all the notations introduced before.

\begin{lem}\em(\cite[IV.3.12]{gmt})\em\label{gmtiv312}
Let $x\in X$. We have:
\begin{itemize}
\item[(i)] If $Z_x$ is an irreducible component of $X_x$ of dimension $d$, its complexification $\widetilde{Z}_x$ is an irreducible component of $\widetilde{X}_x$. If $y\in Y$ satisfies $\pi(Y_y)=\widetilde{Z}_x$, then $y\in Y^{\widehat{\sigma}}$ and $\dim_\R(Y^{\widehat{\sigma}}_y)=d$. 
\item[(ii)] Let $y\in Y^{\widehat{\sigma}}$ be such that $\pi(y)=x$ and $\dim_\R(\pi(Y_y)\cap X_x)<d$. Then $\dim_\R(Y^{\widehat{\sigma}}_y)<d$. In particular, if $\dim_\R(X_x)<d$, we have $\dim_\R(Y^{\widehat{\sigma}}_y)<d$ for all $x\in\pi^{-1}(x)\cap Y^{\widehat{\sigma}}$.
\end{itemize}
\end{lem}

\begin{lem}\em(\cite[IV.3.13]{gmt})\em\label{gmtiv313}
Let $x\in X$ be a point such that all the irreducible components of $X_x$ have dimension $d$. The germ $X_x$ is coherent if and only if there exists a invariant open neighborhood $V$ of $x$ in $\widetilde{X}$ such that
\begin{itemize}
\item[(i)] $\pi^{-1}(X\cap V)=Y^{\widehat{\sigma}}\cap\pi^{-1}(V)$.
\item[(ii)] $\dim_\R(Y^{\widehat{\sigma}}_z)=d$ for every $z\in\pi^{-1}(a)$ and $a\in X\cap V$.
\end{itemize}
\end{lem}

\subsection{Points of non-coherence for maximal dimension outside a prescribed set}\label{comment}\setcounter{paragraph}{0}
Let $R\subset X$ be a closed ${\mathcal A}(\widetilde{X}^\sigma)$-definable $C$-semianalytic set. Let $Y'$ be the union of the irreducible components of $Y$ of dimension strictly smaller than $d$. Define: 
\begin{align*}
&Y^{\widehat{\sigma}}_{(d)}:=\{y\in Y^{\widehat{\sigma}}:\ \dim_{\R}(Y^{\widehat{\sigma}})=d\}=\cl(Y^{\widehat{\sigma}}\setminus(\Sing(Y^{\widehat{\sigma}})\cup Y')),\\ 
&C_1:=\pi^{-1}(X)\setminus Y^{\widehat{\sigma}},\quad C_2:=Y^{\widehat{\sigma}}\setminus Y^{\widehat{\sigma}}_{(d)},\\
&A_i:=\cl(C_i)\cap\cl(Y^{\widehat{\sigma}}_{(d)}\setminus\pi^{-1}(R))\subset Y^{\widehat{\sigma}}\quad\text{for $i=1,2$}.
\end{align*} 
We claim: \em $\pi(A_i)\subset X$ is a $C$-semianalytic set\em.
\begin{proof}
The $C$-semianalytic set $Y^{\widehat{\sigma}}\setminus(\Sing(Y^{\widehat{\sigma}})\cup Y')$ is ${\mathcal A}(Y^{\widehat{\sigma}})$-definable, so by Proposition \ref{salvation} $Y^{\widehat{\sigma}}_{(d)}=\cl(Y^{\widehat{\sigma}}\setminus(\Sing(Y^{\widehat{\sigma}})\cup Y'))$ is ${\mathcal A}(Y^{\widehat{\sigma}})$-definable. By Proposition \ref{salvation} the same happens with $\cl(C_2)$. As $R$ is an ${\mathcal A}(\widetilde{X}^\sigma)$-definable $C$-semianalytic set, then $\pi^{-1}(R)\cap Y^{\widehat{\sigma}}$ is an ${\mathcal A}(Y^{\widehat{\sigma}})$-definable $C$-semianalytic set. By Proposition \ref{salvation} the same happens with $\cl(Y^{\widehat{\sigma}}_{(d)}\setminus\pi^{-1}(R))$. By Theorem \ref{properint}(i) and (ii) $\pi(A_i)\subset X$ is a $C$-semianalytic set for $i=1,2$, as required.
\end{proof}

Define 
\begin{equation}
N_d(X,R):=\pi(A_1)\cup\pi(A_2)\subset X,
\end{equation}
which is a $C$-semianalytic set. We keep all previous notations in the following results.

\begin{lem}\label{dimnk}
We have: 
\begin{itemize}
\item[(i)] $\dim_\R(C_1)\leq d-1$ and $\dim_\R(C_2)\leq d-2$.
\item[(ii)] $\dim_\R(\pi(A_1))\leq d-2$ and $\dim_\R(\pi(A_2))\leq d-3$.
\end{itemize}
\end{lem}
\begin{proof}
We claim: \em If $y\in C_1$, it holds $\dim_\R(C_{1,y})\leq d-1$\em. Consequently, $\dim_\R(C_1)\leq d-1$

If $y\in C_1$, we have $\widehat{\sigma}(y)\neq y$ and $\widehat{\sigma}(Y_y)=Y_{\widehat{\sigma}(y)}$. Observe that $x:=\pi(y)=\pi(\sigma(y))$, so $\pi(Y_y)\cup\pi(Y_{\sigma(y)})\subset\widetilde{X}_x$ and $\pi(Y_y)$, $\pi(Y_{\sigma(y)})$ are two different irreducible components of $\widetilde{X}_x$. As these irreducible components are conjugated, 
$$
X_x\cap\pi(Y_y)=X_x\cap\pi(Y_y)\cap\sigma(\pi(Y_y))=X_x\cap\pi(Y_y)\cap\pi(Y_{\sigma(y)}),
$$
so $\dim_\R(X_x\cap\pi(Y_y))\leq d-1$. Thus, $C_{1,y}\subset\pi^{-1}(X)_y\subset Y_y$ has dimension $\leq d-1$ because $\pi$ is proper and has finite fibers.

As $Y$ is a normal complex analytic space, 
$$
\dim_\C(\Sing(Y))\leq\dim_\C(Y)-2=d-2. 
$$
Consequently, as $\Sing(Y^{\widehat{\sigma}})\subset\Sing(Y)\cap Y^{\widehat{\sigma}}$, we have $\dim_\R(\Sing(Y^{\widehat{\sigma}}))\leq d-2$. As $C_2\subset\Sing(Y^{\widehat{\sigma}})$, also $\dim_\R(C_2)\leq d-2$. 

Observe that $C_i\cap A_i\subset C_i\cap Y^{\widehat{\sigma}}_{(d)}=\varnothing$, so $A_i\subset\cl(C_i)\setminus C_i$. Thus, $\dim_\R(A_i)<\dim_\R(C_i)$ for $i=1,2$ (use \cite[VIII.2.11]{abr}). We conclude $\dim_\R(A_1)\leq d-2$ and $\dim_\R(A_2)\leq d-3$.

As $\pi(A_1)$ and $\pi(A_2)$ are $C$-semianalytic subsets of $M$ and they are the images of $A_1$ and $A_2$ under a proper analytic map with finite fibers, $\dim_\R(\pi(A_1))\leq d-2$ and $\dim_\R(\pi(A_2))\leq d-3$, as required.
\end{proof}

\begin{prop}\label{cluenk}
A point $x\in X$ belongs to $N_d(X,R)$ if and only if the germ $X_x$ has a non-coherent irreducible component $T_x$ of dimension $d$ such that $\dim(T_x\setminus R_x)=d$.
\end{prop}
\begin{proof}
The proof is conducted in several steps. 

\paragraph{} We prove first the `only if' part of the statement. To that end, it is enough to show the following: \em Let $x\not\in N_d(X,R)$ be such that the germ $X_x$ has an irreducible component $B_x$ of dimension $d$ and $\dim(B_x\setminus R_x)=d$. Then, $B_x$ is coherent\em. 

By Lemma \ref{gmtiv312}(i) it holds that the complexification $\widetilde{B}_x$ of $B_x$ is an irreducible component of $\widetilde{X}_x$. In addition, there exists a point $y\in Y^{\widehat{\sigma}}$ such that $\pi(Y_y)=\widetilde{B}_x$ and $\dim(Y^{\widehat{\sigma}}_y)=d$. 

\paragraph{} Let us check: $y\in Y^{\widehat{\sigma}}_{(d)}\setminus(\cl(C_1)\cup\cl(C_2))$. 

As $\dim(Y^{\widehat{\sigma}}_y)=d$, it holds $y\in Y^{\widehat{\sigma}}_{(d)}$. We prove next $\dim(Y^{\widehat{\sigma}}_{(d),y}\setminus\pi^{-1}(R)_y)=d$. 

Indeed, $B_x=X_x\cap\pi(Y_y)$, so 
\begin{multline*}
\pi^{-1}(B_x)\cap Y_y=\pi^{-1}(X_x)\cap\pi^{-1}(\pi(Y_y))\cap Y_y=\pi^{-1}(X)_y\cap Y_y\\
=\pi^{-1}(X)_y=(\pi^{-1}(X)\setminus Y^{\widehat{\sigma}})_y\cup(Y^{\widehat{\sigma}}_y\setminus Y^{\widehat{\sigma}}_{(d),y})\cup Y^{\widehat{\sigma}}_{(d),y}=C_{1,y}\cup C_{2,y}\cup Y^{\widehat{\sigma}}_{(d),y}.
\end{multline*}
Consequently,
$$
\pi^{-1}(B_x\setminus R_x)\cap Y_y=(\pi^{-1}(B_x)\cap Y_y)\setminus\pi^{-1}(R)_y\subset C_{1,y}\cup C_{2,y}\cup(Y^{\widehat{\sigma}}_{(d),y}\setminus\pi^{-1}(R)_y)
$$
As $\pi(\pi^{-1}(B_x\setminus R_x)\cap Y_y)=B_x\setminus R_x$, $\dim(B_x\setminus R_x)=d$ and $\pi$ has finite fibers, we deduce 
$$
C_{1,y}\cup C_{2,y}\cup(Y^{\widehat{\sigma}}_{(d),y}\setminus\pi^{-1}(R)_y)
$$
has dimension $d$. By Lemma \ref{dimnk}(i) $\dim(C_{1,y}\cup C_{2,y})\leq d-2$, so $\dim(Y^{\widehat{\sigma}}_{(d),y}\setminus\pi^{-1}(R)_y)=d$.

In particular, $y\in\cl(Y^{\widehat{\sigma}}_{(d)}\setminus\pi^{-1}(R))\subset Y^{\widehat{\sigma}}_{(d)}$. If $y\in\cl(C_i)$, then $y\in A_i$, so $x\in\pi(A_i)\subset N_d(X,R)$, which is a contradiction. Thus, $y\in Y^{\widehat{\sigma}}_{(d)}\setminus(\cl(C_1)\cup\cl(C_2))$. 

\paragraph{}\label{neigh88} Write $\pi^{-1}(x)=\{y_1,\ldots,y_r\}$ and assume $y_1=y$. Let $W_i\subset Y$ be an invariant open neighborhood of $y_i$ and let $V$ be an open neighborhood of $x$ such that $\pi^{-1}(V)=\bigcup_{i=1}^rW_i$ and $W_i\cap W_j=\varnothing$ if $i\neq j$. We may assume that $W_1$ is connected and $W_1\cap(\cl(C_1)\cup\cl(C_2))=\varnothing$. As $W_1$ is closed in $\pi^{-1}(V)$ and $\pi:Y\to\widetilde{X}$ is proper, the restriction $\pi:W_1\to V$ is proper, so by Remmert's Theorem the set $\pi(W_1)$ is a complex analytic subset of $V$. Observe that $\pi:W_1\to\pi(W_1)$ is the normalization of $\pi(W_1)$. The set $B:=\pi(W_1)\cap X$ is a representative of $B_x$. As $\dim(\pi(W_1))=\dim(B)$ and $\pi(W_1)$ is irreducible (recall that $W_1$ is normal and connected), $\pi(W_1)$ is a complexification of $B$.

\paragraph{} We claim: $B^{\widehat{\sigma}}:=\{y\in W_1:\ \widehat{\sigma}(y)=y\}=(\pi|_{W_1})^{-1}(B)$.

Indeed, observe that $B^{\widehat{\sigma}}=Y^{\widehat{\sigma}}\cap W_1$ and 
$$
(\pi|_{W_1})^{-1}(B)=\pi^{-1}(\pi(W_1))\cap\pi^{-1}(X)\cap W_1=\pi^{-1}(X)\cap W_1.
$$
Thus, we have to prove $Y^{\widehat{\sigma}}\cap W_1=\pi^{-1}(X)\cap W_1$. As $W_1\cap\cl(C_1)=\varnothing$, we have $(\pi^{-1}(X)\setminus Y^{\widehat{\sigma}})\cap W_1=\varnothing$, so $Y^{\widehat{\sigma}}\cap W_1=\pi^{-1}(X)\cap W_1$.

\paragraph{} Since $W_1\cap\cl(C_2)=\varnothing$, we have $(Y^{\widehat{\sigma}}\setminus Y^{\widehat{\sigma}}_{(d)})\cap W_1=\varnothing$, so $\dim(Y^{\widehat{\sigma}}_z)=d$ for each $z\in(\pi|_{W_1})^{-1}(a)$ where $a\in B$. By Lemma \ref{gmtiv313} the irreducible component $B_x$ is coherent.

\paragraph{} Next we show the `if' part of the statement. To that end, we prove: \em If all the non-coherent irreducible components $B_x$ of $X_x$ of dimension $d$ satisfy $\dim(B_x\setminus R_x)<d$, then $x\not\in N_d(X,R)$\em, or equivalently, $\pi^{-1}(x)\cap(A_1\cup A_2)=\varnothing$. 

Let $y\in\pi^{-1}(x)$. If $y\not\in Y^{\widehat{\sigma}}_{(d)}$, then $y\not\in A_1\cup A_2$. So let us assume $y\in Y^{\widehat{\sigma}}_{(d)}$. By Lemma \ref{gmtiv312}(ii) we deduce $\dim(\pi(Y_y)\cap X_x)=d$. Consequently, $B_x:=\pi(Y_y)\cap X_x$ is an irreducible analytic germ of dimension $d$ contained in $X_x$. Thus, $B_x$ is an irreducible component of $X_x$ and $\pi(Y_y)$ is the complexification of $B_x$. In particular $\pi(Y^{\widehat{\sigma}}_y)\subset B_x$. We distinguish two cases: 

\noindent{\bf Case 1.} \em $B_x$ is non-coherent\em. Then $\dim(B_x\setminus R_x)<d$, so $\dim(\pi(Y^{\widehat{\sigma}}_y)\setminus R_x)<d$. Denote $E_y:=Y^{\widehat{\sigma}}_{(d),y}\setminus\pi^{-1}(R)_y$. As $\pi$ has finite fibers and $\pi(E_y)\subset\pi(Y^{\widehat{\sigma}}_y)\setminus R_x$, we have $\dim(E_y)<d$. As $R$ is closed and $\dim(E_y)<d$,
$$ 
Y^{\widehat{\sigma}}_{(d),y}=\cl(Y^{\widehat{\sigma}}_{(d),y}\setminus E_y)\subset\cl(\pi^{-1}(R)_y)=\pi^{-1}(R)_y.
$$
Consequently, $Y^{\widehat{\sigma}}_{(d),y}\setminus\pi^{-1}(R)_y=\varnothing$, so $y\not\in\cl(Y^{\widehat{\sigma}}_{(d)}\setminus\pi^{-1}(R))$. We conclude $y\not\in A_1\cup A_2$. 

\noindent{\bf Case 2.} \em $B_x$ is coherent\em. Proceeding as we have done in \ref{neigh88} we find an invariant neighborhood $W_1$ of $y$ in $Y$ such that $\pi(W_1)$ is an irreducible complex analytic subset of an invariant neighborhood $V$ of $x$ in $X$ and $B:=\pi(W_1)\cap X$ is a representative of $B_x$. As $B_x$ is coherent, by Lemma \ref{gmtiv313} we may shrink $V$ and $W_1$ to have 
\begin{itemize}
\item[(1)] $\pi^{-1}(X)\cap W_1=Y^{\widehat{\sigma}}\cap W_1$ and 
\item[(2)] $\dim(Y^{\widehat{\sigma}}_z)=d$ for every $z\in\pi^{-1}(a)$ and $a\in B\cap V$.
\end{itemize}
Condition (1) is equivalent to $(\pi^{-1}(X)\setminus Y^{\widehat{\sigma}})\cap W_1=\varnothing$, so $y\not\in \cl(C_1)$. Condition (2) means $y\not\in\cl(C_2)$. Consequently, $y\not\in\cl(C_1)\cup\cl(C_2)$, so $y\not\in A_1\cup A_2$. 

Thus, $\pi^{-1}(x)\cap(A_1\cup A_2)=\varnothing$, as required. 
\end{proof}

\subsection{Proof of Theorem \ref{ncp0}}
Given a set germ $S_x\subset\R^n_x$, we define its Zariski closure $\ol{S_x}^{\zar}$ as the smallest analytic germ at $x$ that contains $S_x$. Before proving Theorem \ref{ncp0} we state the following easy fact concerning germs that will be used several times in its proof.

\begin{lem}\label{easy}
Let $A_x\subset B_x\subset\R^n_x$ be analytic germs and let $T_x$ be an irreducible component of $B_x$ such that $T_x\subset A_x$. Then $T_x$ is also an irreducible component of $A_x$.
\end{lem}
\begin{proof}
Let $T'_x$ be an irreducible component of $A_x$ such that $T_x\subset T_x'$. As $A_x\subset B_x$, there exists an irreducible component $T_x''$ of $B_x$ such that $T_x\subset T_x'\subset T_x''$. Consequently $T_x=T_x'=T_x''$, so $T_x$ is an irreducible component of $A_x$. 
\end{proof}

\begin{proof}[Proof of Theorem \em\ref{ncp0}]
Observe that by its definition and Proposition \ref{salvation} each $C$-semianalytic set $R_k$ is ${\mathcal A}(\widetilde{X}^\sigma)$-definable. Consequently, statement (i) follows from \ref{comment} and Lemma \ref{dimnk}(ii), so it remains to prove statements (ii) and (iii). We have to show: \em $N(X)=\bigcup_{k=2}^dN_k(Z_k,R_k)$ where \em
\begin{itemize}\em
\item[$\bullet$] $Z_k:=\bigcup_{Z\in{\mathfrak F} _k}Z$, 
\item[$\bullet$] ${\mathfrak F} _k$ is the collection of all the irreducible $C$-semianalytic subsets $Z$ of $M$ of dimension $k$ that are an irreducible component of $\Sing_\ell(X)$ for some $\ell\geq0$,
\item[$\bullet$] $R_k:=\bigcup_{j=k+1}^dZ_{j,(j)}$ where $Z_{j,(j)}:=\{z\in Z_j:\ \dim_{\R}(Z_j)=j\}=\cl(Z_j\setminus\Sing(Z_j))$.
\end{itemize}
In addition, we have to prove: \em $\bigcup_{k=j}^dN_k(Z_k,R_k)$ is the set of points of $X$ such that the germ $X_x$ has a non-coherent irreducible component of dimension $\geq j$\em.

The proof is conducted in several steps.\setcounter{paragraph}{0}

\paragraph{}\label{ze} Let $x\in N(X)$ and let $T_x$ be an irreducible component of $X_x$ of dimension $e$ that it is non-coherent. Recall that $e\geq2$ because $C$-analytic curves are coherent. We claim: \em $T_x$ is an irreducible component of $Z_{e,x}$\em.

Let $Z$ be an irreducible component of $X$ such that $T_x\subset Z_x$. As $Z_x\subset X_x$, we have \em $T_x$ is an irreducible component of $Z_x$ \em by Lemma \ref{easy}. 
As $Z=\bigcup_{\ell\geq0}\Reg(\Sing_\ell(Z))$,
$$
Z_x=\bigcup_{\ell\geq0}\ol{(\Reg(\Sing_\ell(Z)))_x}^{\zar}.
$$
As $T_x$ is an irreducible component of $Z_x$, there exists $\ell\geq1$ such that 
$$
T_x\subset\ol{(\Reg(\Sing_\ell(Z)))_x}^{\zar}\subset Z_x. 
$$
Thus, $T_x$ is by Lemma \ref{easy} an irreducible component of $\ol{(\Reg(\Sing_\ell(Z)))_x}^{\zar}$. Since all the irreducible components of $\ol{(\Reg(\Sing_\ell(Z)))_x}^{\zar}$ have the same dimension and $\dim(T_x)=e$, we deduce $\dim(\Sing_\ell(Z))=e$. By Lemma \ref{easy} there exists an irreducible component $Z'$ of $\Sing_\ell(Z)$ (of dimension $e$) such that $T_x$ is an irreducible component of $Z'_x$. Observe that $Z'\in{\mathfrak F} _e$, so $T_x$ is an irreducible component of $Z_{e,x}$. 

\paragraph{}\label{ne} We claim: $x\in N_e(Z_e,R_e)$ or equivalently by Proposition \ref{cluenk} \em the germ $Z_{e,x}$ has a non-coherent irreducible component $B_x$ of dimension $e$ such that $\dim(B_x\setminus R_{e,x})=e$\em. It is enough to check: $\dim(T_x\setminus R_{e,x})=e$. 

Otherwise, $\dim(T_x\setminus R_{e,x})<e$, so
$$
T_x\subset\ol{R_{e,x}}^{\zar}=\bigcup_{j=e+1}^d\ol{Z_{j,(j),x}}^{\zar}.
$$
Consequently, there exists $e+1\leq j\leq d$ such that $T_x\subset\ol{Z_{j,(j),x}}^{\zar}\subset X_x$. By Lemma \ref{easy} $T_x$ is an irreducible component of $\ol{Z_{j,(j),x}}^{\zar}$, which is a contradiction because all the irreducible components of $\ol{Z_{j,(j),x}}^{\zar}$ have dimension $j>e$. We conclude $\dim(T_x\setminus R_{e,x})=e$.

\paragraph{}\label{geql} Thus, we have shown $N(X)\subset\bigcup_{k=2}^dN_k(Z_k,R_k)$. We prove now the converse inclusion $\bigcup_{k=2}^dN_k(Z_k,R_k)\subset N(X)$. Let $x\in N_\ell(Z_\ell,R_\ell)$ for some $2\leq\ell\leq d$ and let us show: \em $X_x$ has a non-coherent irreducible component of dimension $\geq\ell$\em. In particular, $x\in N(X)$.

As $x\in N_\ell(Z_\ell,R_\ell)$, the germ $Z_{\ell,x}$ has an irreducible component $T_x$ of dimension $\ell$ that is non-coherent and such that $\dim(T_x\setminus R_{\ell,x})=\ell$. As $Z_{\ell,x}\subset X_x$, there exists an irreducible component $A_x$ of $X_x$ that contains $T_x$. If $T_x=A_x$, then $A_x$ is non-coherent. Otherwise, $\ell=\dim(T_x)<\dim(A_x)=j$. As
$$
X_x=\bigcup_{k=0}^d\ol{Z_{k,(k),x}}^{\zar},
$$
we deduce by Lemma \ref{easy} that $A_x$ is an irreducible component of some $\ol{Z_{k,(k),x}}^{\zar}$. As all the irreducible components of $\ol{Z_{k,(k),x}}^{\zar}$ have dimension $k$, we conclude $j=k$. Consequently, $A_x$ is an irreducible component of $\ol{Z_{j,(j),x}}^{\zar}$.

Let us assume by contradiction that $A_x$ is coherent. Then $A_x\subset\ol{Z_{j,(j),x}}^{\zar}\subset Z_{j,x}$ is pure dimensional (because it is coherent), so $T_x\subset A_x\subset Z_{j,(j),x}\subset R_{\ell,x}$, which contradicts the fact $\dim(T_x\setminus R_{\ell,x})=\ell$. Consequently, $A_x$ is non-coherent, so $X_x$ has a non-coherent irreducible component and $x\in N(X)$, as required.

\paragraph{} Finally, by \ref{ze} and \ref{ne} we deduce that if $X_x$ has a non-coherent irreducible component of dimension $j$, then $x\in N_j(Z_j,R_j)$ while by \ref{geql} we get that if $x\in N_j(Z_j,R_j)$, then $X_x$ has an irreducible component of dimension $\geq j$. Consequently, statement (ii) holds. 
\end{proof}

We finish this section with two examples that illustrate some key facts of the proofs above. 

\begin{examples}
(i) In \cite{abt} it is shown that the set of points of non-coherence of the $C$-analytic set $X:=\{x^3-x^2wz-wy^2=0\}$ is
$$
N(X)=\{x=0,y=0,z=0\}\cup\{x=0,y=0,w=0,z\geq0\}, 
$$
which is a $C$-semianalytic set but not a $C$-analytic set.

(ii) Let $X_1:=\{(x^2-(z+1)y^2)^2z-u^2=0\}\subset\R^4$ and $X_2:=\{u=0\}$. Let us prove 
$$
N(X_1)=\{x^2-y^2=0,z=0,u=0\}\cup\{(0,0,-1,0)\}
$$ 
while $N(X_1\cup X_2)=\{x^2-y^2=0,z=0,u=0\}$.

We compute first $N(X_1)$. Let $Z_1:=\{(x^2-(z+1)y^2)^2z-u^2=0\}\subset\C^4$ be a complexification of $X_1$. It holds that
$$
\pi:\C^3\to\C^4,\ (x,y,v)\mapsto(x,y,v^2,v(x^2-(v^2+1)y^2)),
$$
is the normalization of $Z_1$. The singular locus of $Z_1$ is 
$$
\Sing(Z_1)=\{x^2-(z+1)y^2=0,u=0\}.
$$
It holds
$$
\pi^{-1}(X_1)=\R^3\sqcup\{(s\sqrt{1-t^2},s,it):\ 0<|t|\leq1\}\sqcup\{(0,0,it):\ |t|>1\}.
$$
Observe that 
\begin{align*}
&T_1:=\pi(\R^3)=X_1\cap\{z\geq0\},\\ 
&T_2:=\pi(\{(\pm s\sqrt{1-t^2},s,it):\ 0<|t|\leq1\})=\{(\pm s\sqrt{1-t^2},s,-t^2,0):\ 0<|t|\leq1\},\\
&T_3:=\pi(\{(0,0,it):\ |t|>1\})=\{(0,0,-t^2,0):\ |t|>1\}.
\end{align*}
Thus, $T_1$ is the set of points of $X_1$ of maximal dimension. By Theorem \ref{ncp0} 
\begin{multline*}
\pi(\R^3\cap(\{(\pm s\sqrt{1-t^2},s,-t^2,0):\ 0<|t|\leq1\}\cup\{(0,0,-t^2,0):\ |t|>1\}))\\
=\{x-y=0,z=0,u=0\}\cup\{x+y=0,z=0,u=0\}
\end{multline*}
is the set of points of $X_1$ that have a non-coherent irreducible component of dimension $3$.

To find the set of points of $X_1$ that have a non-coherent irreducible component of smaller dimension we have to look at $T_2\cup T_3=\{x^2-(z+1)y^2,z<0\}$, which is an open subset a classical Whitney umbrella. This set has only the point $(0,0,-1,0)$ as its unique non-coherence point. Consequently, 
$$
N(X_1)=\{x^2-y^2=0,z=0,u=0\}\cup\{(0,0,-1,0)\}.
$$

On the other hand, $N(X_1\cup X_2)=\{x^2-y^2=0,z=0,u=0\}$ because at these points of $X_1\cup X_2$ the corresponding germ has a non-coherent irreducible component of dimension $3$ while at the point $(0,0,-1,0)$ the unique irreducible component is $\{u=0\}$, which is coherent.
\end{examples}

\section{Subanalytic sets as proper images of basic $C$-semianalytic sets}\label{s6}
\setcounter{paragraph}{0}

We begin with some examples concerning the properties of the images of $C$-semianalytic sets under analytic maps.

\begin{example}[Image under a proper analytic map with finite fibers of a $C$-analytic set]
Consider the compact analytic set $S:=\{f=0\}$ introduced in Example \ref{counterexamples}(ii). Observe that $S\setminus\{(0,0,\pm1)\}$ is a $C$-analytic subset of $\R^3\setminus\{z=\pm1\}$ and $S_{(2)}\subset\{g:=\frac{1}{4}-x^2+y^2+z^2\geq0\}$. Define
$$
X_1:=\{(x,y,z,t)\in\R^4:\ f=0, t^2-(\tfrac{1}{4}-x^2-y^2-z^2)=0\},
$$
which is a compact $2$-dimensional $C$-analytic subset of the sphere $\sph_1:=\{x^2+y^2+z^2+t^2=\tfrac{1}{4}\}$. Define $\pi:\R^4\to\R^3, (x,y,z,t)\mapsto(x,y,z)$.

The map $\pi|_{\sph_1}$ is proper, analytic, has finite fibers and satisfies $\pi(X_1)=S\cap\{g\geq0\}$. Let $\sph_2:=\{x^2+(y-1)^2+z^2+(t-2)^2=1\}$ and $X_2:=\sph_2\cap\{x=0,t=2\}$. The map $\pi|_{\sph_2}$ is proper, analytic, has singleton fibers and satisfies $\pi(X_2)=S\cap\{x=0\}=\{x=0,(y-1)^2+z^2=1\}$. Of course $X:=X_1\cup X_2$ is a $C$-analytic subset of $M:=\sph_1\sqcup\sph_2$. Thus, $\pi|_{M}$ is proper, analytic, has finite fibers and satisfies $\pi(X_1\cup X_2)=S$, which is semianalytic but not $C$-semianalytic.
\end{example}

\begin{example}[Osgood: Subanalytic set that is not a semianalytic set]\label{osgood}
Let 
$$
f:\R^2\to\R^3,\ (x,y)\mapsto(x,xy,xe^y). 
$$
Then $S:=f(\{x^2+y^2\leq\veps^2\})$ is subanalytic but it is not semianalytic.

Observe that $S$ is subanalytic because it is the image of a compact semianalytic set under an analytic map. Let us prove next that $S$ is not semianalytic. We claim: \em If $G(u,v)\in\R[[u,v]]$ is a formal power series such that $G(x,xy,xe^y)=0$, then $G=0$\em.

Write $G(u,v,w)=\sum_{j\geq0}G_j(u,v,w)$ where $G_j(u,v,w)$ is a homogeneous polynomial of degree $j$. Then
$$
0=G(x,xy,xe^y)=\sum_{j\geq0}G_j(x,xy,xe^y)=\sum_{j\geq0}x^jG_j(1,y,e^y).
$$
Therefore, $G_j(1,y,e^y)=0$ for each $j\geq0$, so each $G_j=0$ and $G=0$.

Consequently, the smallest real analytic set containing (the germ at the origin of) $S$ is the whole $\R^2$, so $S$ is not semianalytic.
\end{example}

\begin{example}[Image under an analytic map of a $C$-analytic set]\label{notsub}
The image of a $C$-analytic set under an analytic map is not in general a subanalytic set. Let $X:=\bigcup_{k\geq1}\{(\frac{1}{k},k)\}\subset\R^2$ and let $\rho:\R^2\to\R,\ (x,y)\mapsto x$ be the projection onto the first coordinate. Then $S=\rho(X)=\bigcup_{k\geq1}\{\frac{1}{k}\}\subset\R$ is not subanalytic.

Indeed, suppose $S$ is subanalytic. Then $S':=S\setminus\{0\}$ is also subanalytic and there exists a neighborhood $U$ of the origin such that $S'\cap U$ is the projection of a relatively compact semianalytic set $A\subset M\times N$ where $N$ is a real analytic manifold. By \cite[2.7]{bm} the connected components of $A$ are finitely many. Consequently, $\pi(A)=S\cap U$ has finitely many connected components, a contradiction.
\end{example}

\subsection{Proof of Theorem \ref{sub}}
Recall first two relevant results \cite[3.12, 5.1]{bm}. Theorem \ref{sub} can be understood as a kind of global version of \cite[3.12]{bm}. Following \cite[3.5]{bm} the dimension of a subanalytic set is the highest of the dimensions of its smooth points.

\begin{prop}[{\cite[3.12]{bm}}]
Let $M$ be a real analytic manifold and let $S\subset M$ be a closed subanalytic set. Then each point of $S$ admits a neighborhood $U$ such that $S\cap U=\pi(A)$ where $A$ is a closed analytic subset of $U\times\R^q$ for some $q$, $\dim(A)=\dim(S\cap U)$ and $\pi|_{A}$ is proper (where $\pi:U\times\R^q\to U$ is the projection).
\end{prop}

\begin{thm}[Uniformization Theorem, {\cite[5.1]{bm}}]\label{ut}
Let $X$ be a closed analytic subset of $M$. Then there exists a real analytic manifold $N$ (of the same dimension as $X$) and a proper real analytic map $p:N\to M$ such that $p(N)=X$.
\end{thm}

We are ready to prove Theorem \ref{sub}.

\begin{proof}[Proof of Theorem \em \ref{sub}]
The implication (ii) $\Longrightarrow$ (iii) is immediate. Let us prove (iii) $\Longrightarrow$ (i). We have to show: \em each point of $M$ admits a neighborhood $U$ such that $S\cap U$ is a projection of a relatively compact semianalytic set\em. 

Let $\Gamma_f$ be the graph of $f$ and let $\pi:M\times N\to N$ be the projection onto the second factor. Let $T':=\Gamma_f\cap(T\times N)$ and $C:=\Gamma_f\cap(\cl(T)\times N)$, which are semianalytic subset of $M\times N$. 

\paragraph{} We claim: $C=\cl(T')$. 

Only the inclusion $C\subset\cl(T')$ requires a comment. Pick $(x,f(x))\in C$ and let $U\times V$ be a neighborhood of $(x,f(x))$ in $M\times N$. As $f$ is continuous, we may assume $f(U)\subset V$. As $x\in\cl(T)$, there exists $x'\in T\cap U$, so $(x',f(x'))\in(U\times V)\cap(\Gamma_f\cap(T\times N))$. Consequently $(x,f(x))\in\cl(T')$.

\paragraph{} Let us prove: \em The restriction map $\pi|_{C}:C\to N$ is proper\em. 

Indeed, $\Gamma_f$ is a real analytic submanifold of $M\times N$ and it is analytically diffeomorphic to $M$ via the restriction to $\Gamma_f$ of the projection $\rho:M\times N\to M$. Let $K$ be a compact subset of $N$ and observe that 
$$
f^{-1}(K)\cap\cl(T)=\rho(\Gamma_f\cap(\cl(T)\times N)\cap(M\times K))=\rho(C\cap\pi^{-1}(K))=\rho(\pi|_{C}^{-1}(K)). 
$$
As $f|_{\cl(T)}:\cl(T)\to N$ is proper, $f^{-1}(K)\cap\cl(T)$ is compact. As $\rho|_{\Gamma_f}:\Gamma_f\to M$ is an analytic diffeomorphism, $\pi|_{C}^{-1}(K)$ is compact. Consequently, $\pi|_{C}$ is proper.

\paragraph{} Let $y\in N$ and let $U$ be an open semianalytic neighborhood of $y$ in $N$ such that $K:=\cl(U)$ is compact. As $\pi|_C$ is proper, $\pi|_C^{-1}(K)$ is compact. As $f(T)=S$, it holds $\pi(T')=S$, so $\pi(T'\cap\pi^{-1}(U))=S\cap U$. It only remains to prove: \em $A:=T'\cap\pi^{-1}(U)$ is a relatively compact semianalytic set\em. 

As $T'$ is a semianalytic subset of $M\times N$ and $U$ is a semianalytic subset of $N$, we have that $A$ is a semianalytic subset of $M\times N$. To prove that $A$ is relatively compact, we only need to show that it is contained in a compact subset of $M\times N$. Indeed,
$$
A=T'\cap\pi^{-1}(U)\subset C\cap\pi^{-1}(U)=\pi|_C^{-1}(U)\subset\pi|_C^{-1}(K),
$$
which is a compact set because $\pi|_C$ is proper.

We prove next (i) $\Longrightarrow$ (ii). Let $S$ be a subanalytic subset of $N$. 

\noindent{\bf Step 1.} {\em Local construction.} Fix $x_0\in N$ and denote $n:=\dim(N)$. In this step we prove: \em there exist and open neighborhood $U\subset N$ of $x_0$, a compact real analytic manifold $M\subset\R^{2n+1}$, an analytic function $g\in\an(M)$ and an analytic map $\pi:M\to N$ such that 
$$
\pi(\{g>0\}\cap\pi^{-1}(U))=S\cap U.
$$\em

As $S$ is subanalytic, there exist an open neighborhood $U\subset N$ of $x_0$, a real analytic manifold $N'$ and a relatively compact semianalytic subset $A$ of $N\times N'$ such that $S\cap U=\pi_1(A)$ where $\pi_1:N\times N'\to N$ is the projection onto the first factor. We can suppose $\dim(S\cap V)=\dim(S\cap U)$ for each open neighborhood $V\subset U$ of $x_0$. 

\paragraph{}\label{ab} We may assume: \em $\dim(A)=\dim(S\cap U)$ and $\pi_1^{-1}(x_0)\cap A$ is a finite set\em.

By \cite[3.6]{bm} there exist finitely many smooth semianalytic subsets $B_k$ of $A$ such that
\begin{itemize}
\item $S\cap U=\pi_1(A)=\pi_1(\bigcup_kB_{k})$.
\item For each $B_{k}$ the restriction $\pi_1|_{B_k}:B_k\to N$ is an immersion.
\end{itemize}
This means that there exists a relatively compact semianalytic subset $B:=\bigcup_kB_k$ of $N\times N'$ of the same dimension as $S\cap U$ such that $\pi_1(B)=S\cap U$ and $\pi_1^{-1}(x_0)\cap B$ is a finite set. After substituting $A$ by $B$, we are under the hypothesis of \ref{ab}.

\paragraph{} After shrinking $U$ and using that $\pi_1^{-1}(x_0)\cap A$ is finite, \em there exists an open neighborhood $W$ of $\pi_1^{-1}(x_0)$ in $N\times N'$ and finitely analytic functions $f_i,g_{ij}\in\an(W)$ such that $A=\bigcup_{i=1}^r\{f_i=0,g_{i1}>0,\ldots,g_{is}>0\}$\em.

\paragraph{} Let $h\in\an(W)$ be such that $\pi_1^{-1}(x_0)\cap A\subset\{h>0\}$ and $\{h\geq0\}$ is compact. After shrinking $U$ we assume $S\cap U=\pi_1(A\cap\{h>0\}\cap\pi_1^{-1}(U))$. 

Now, we transform inequalities into equalities. For each $i=1,\ldots,r$ define
\begin{multline*}
X_i:=\{(x,y,z)\in W\times\R^{s+2}:\\
f_i(x,y)=0,z_1^2-g_{i1}(x,y)=0,\ldots,z_s^2-g_{is}(x,y)=0,z_{s+1}^2-h(x,y)=0,z_{s+2}=i\}
\end{multline*}
Let $X:=\bigcup_{i=1}^rX_i$, which is a compact analytic subset of $W\times\R^{s+2}$ (recall that $\{h\geq0\}$ is compact) of the same dimension as $A$, so $\dim(X)\leq\dim(N)$. Let $\pi_2:N\times N'\times\R^{s+2}\to N\times N'$ be the projection onto the first two factors. Let $g_i:=\prod_{j=1}^sg_{ij}$ and $g':=h\prod_{i=1}^r(g_i^2+(z_{s+2}-i)^2)$. Notice that 
\begin{align*}
&\pi_2(X\setminus\{g'=0\}\cap\pi_2^{-1}(\pi_1^{-1}(U)))=A\cap\{h>0\}\cap\pi_1^{-1}(U),\\
&\pi_1(\pi_2(X\setminus\{g'=0\}\cap\pi_2^{-1}(\pi_1^{-1}(U))))=S\cap U.
\end{align*} 
By Theorem \ref{ut} there exists a compact real analytic manifold $M$ (of the same dimension as $X$) and a proper real analytic map $p:M\to W$ such that $p(M)=X$. 

\paragraph{} By Whitney's immersion theorem for the analytic case \cite[2.15.12]{n2} we can embed $M$ in $\R^{2n+1}$ as a closed analytic submanifold where $n=\dim(N)\geq\dim(M)$. Write $g:=(g')^2\circ p$ and $\pi:=\pi_1\circ\pi_2\circ p:M\to N$. We have $\pi(\{g>0\}\cap\pi^{-1}(U))=S\cap U$.

\noindent{\bf Step 2.} {\em Global construction.} By Lemma \ref{cover} there exists a countable locally finite open refinement $\{V_j\}_{j\geq1}$ of $\{U^x\}_{x\in N}$ such that each $V_j=\{h_j>0\}$ is a open $C$-semianalytic subset of $N$ where $h_j\in\an(N)$. We have seen above that for each $j\geq1$ there exists a compact real analytic submanifold $M_j\subset\R^{2n+1}\times\{j\}\subset\R^{2n+2}$, an analytic function $g_j\in\an(M_j)$ and an analytic map $\pi_j:M_j\to N$ such that 
$$
S\cap V_j=\pi_j(\{g_j>0\}\cap\pi_j^{-1}(V_j))=\pi_j(\{g_j>0,(h_j\circ\pi_j)>0\}).
$$ 
Consider the real analytic manifold $M:=\bigsqcup_{j\geq1}M_j\subset\R^{2n+2}$, whose connected components are all compact. Let $g,h\in\an(M)$ be given by $g|_{M_j}=g_j$ and $h|_{M_j}=h_j\circ\pi_j$. Consider the analytic map $\pi:M\to N$ such that $\pi|_{M_j}=\pi_j$ and define $T:=\{g>0,h>0\}$, which is a basic $C$-semianalytic set. We have 
$$
\pi(T)=\pi(\{g>0,h>0\})=\bigcup_{j\geq1}\pi_j(\{g_j>0,(h_j\circ\pi_j)>0\})=\bigcup_{j\geq1}S\cap V_j=S.
$$

\paragraph{}It only remain to check: \em $\pi|_{\cl(T)}:\cl(T)\to N$ is proper\em. 

Let $K_0$ be a compact subset of $N$ and denote $K:=K_0\cap\cl(S)$. As $S=\bigcup_{j\geq1}S\cap V_j$ and the family $V_j$ is locally finite, $\cl(S)=\bigcup_{j\geq1}\cl(S\cap V_j)$ and the family $\{\cl(S\cap V_j)\}_{j\geq1}$ is locally finite. As $K$ is compact, we may assume $K\cap\cl(S\cap V_j)=\varnothing$ for $j\geq\ell$. As the family $\{M_j\}_{j\geq1}$ is locally finite and $M_j\cap M_k=\varnothing$ if $j\neq k$, we have $\cl(T)\cap M_j=\cl(T\cap M_j)$. In addition, $\pi(T\cap M_j)=S\cap V_j$. We claim: \em $\pi^{-1}(K)\cap\cl(T)\cap M_j=\varnothing$ for $j\geq\ell$\em. 

Suppose by contradiction that there exists $x\in\pi^{-1}(K)\cap\cl(T\cap M_j)$ for some $j\geq\ell$. Thus, as $\pi$ is continuous and $j\geq\ell$, 
\begin{multline*}
\pi(x)\in K\cap\pi(\cl(T)\cap M_j)=K\cap\pi(\cl(T\cap M_j))\\
\subset K\cap\cl(\pi(T\cap M_j))=K\cap\cl(S\cap V_j)=\varnothing,
\end{multline*}
which is a contradiction.

As $\pi^{-1}(K)\subset\bigcup_{j=1}^{\ell-1}\cl(T)\cap M_j$ and each $M_j$ is compact, we conclude that $\pi^{-1}(K)$ is compact, so $\pi|_{\cl(T)}:\cl(T)\to N$ is proper, as required.
\end{proof}

\end{document}